\documentclass{stylefile}
\usepackage{enumerate}
\usepackage{scalerel}
\usepackage[shortlabels]{enumitem}

\usepackage{accents}
\usepackage{extarrows}
\usepackage{bbm}
\usepackage[outline]{contour}
\usepackage{color}
\definecolor{midgrey}{RGB}{150,173,180}

\makeatletter
\newcommand*\bigcdot{\mathpalette\bigcdot@{.5}}
\newcommand*\bigcdot@[2]{\mathbin{\vcenter{\hbox{\scalebox{#2}{$\m@th#1\bullet$}}}}}
\makeatother

\newcommand\blfootnote[1]{
  \begingroup
  \renewcommand\thefootnote{}\footnote{#1}
  \addtocounter{footnote}{-1}
  \endgroup
}
\usepackage{comment}

\definecolor{mydarkorange}{RGB}{0,0,0}
\definecolor{gold}{rgb}{0,0,0}
\definecolor{grey}{RGB}{0,0,0}
\definecolor{myorange}{RGB}{0,0,0}
\definecolor{mydarkorange}{RGB}{0,0,0}
\definecolor{mylightblue}{RGB}{0,0,0}
\definecolor{myyellow}{RGB}{0,0,0}
\definecolor{purple}{RGB}{0,0,0}
\definecolor{brown}{RGB}{0,0,0}
\definecolor{myblue}{RGB}{0,0,0}
\definecolor{mygreen}{RGB}{0,0,0}

\usepackage{graphicx}

\usepackage{amsmath}
\usepackage{mathtools}
\usepackage{amssymb}

\newtheorem{theorem}{Theorem}[section]
\newtheorem{lemma}[theorem]{Lemma}

\theoremstyle{definition}
\newtheorem{definition}[theorem]{Definition}
\newtheorem{example}[theorem]{Example}

\theoremstyle{remark}
\newtheorem{remark}[theorem]{Remark}

\numberwithin{equation}{section}

\newcommand*\cleartoleftpage{%
  \clearpage
  \ifodd\value{page}\hbox{}\newpage\fi
}

\newcommand{\te}{}
\newcommand{\teALONE}{{\cdot}}
\newcommand{\g}                 [2] {{#1}\te{#2}}
\newcommand{\teu}           {\mathbin{\teALONE_u}} 
 
\newcommand{\claim}{}

\newcommand{\tev}           {\mathbin{\teALONE_v}}

\newcommand{\komp}      {{^\prime}}

\newcommand{\kompM}[1]{^{{\prime^{\mkern-4mu^{_{#1}}}}}}
\newcommand{\nega}      [1] {{#1}\komp}
\newcommand{\negaM}      [2] {{#2}\kompM{#1}}

\newcommand{\ite}[1]{\mathbin{\rightarrow_{#1}}}

\newcommand{\gstar}               [2] {{#1} \mathbin{\star} {#2}}

\newcommand{\gteu}         [2] {{#1} \mathbin{\teALONE_u} {#2}}

\newcommand{\gteM}       [3] {{#2} \mathbin{\teALONE_{#1}} {#3}}
\newcommand{\gtev}         [2] {{#1} \mathbin{\tev} {#2}}

\newcommand{\res}               [3] {{#2}\mathbin{\ite{#1}}{#3}}

\begin{document}

\title[Group representation for a class of residuated chains]{Group representation for even and odd involutive commutative residuated chains}
\thanks{The present scientific contribution was supported by the GINOP 2.3.2-15-2016-00022 grant
and the Higher Education Institutional Excellence Programme 20765-3/2018/FEKUTSTRAT of the Ministry of Human Capacities in Hungary.}


\begin{abstract}
For odd and for even involutive, commutative residuated chains a representation theorem is presented in this paper by means of direct systems of abelian $o$-groups equipped with further structure. 
This generalizes the corresponding result of J. M. Dunn about finite Sugihara monoids.

\end{abstract}

\author{S\'andor Jenei}
\address{Institute of Mathematics and Informatics, University of P\'ecs, Ifj\'us\'ag u. 6, H-7624 P\'ecs, Hungary, jenei@ttk.pte.hu}

\date{}

\maketitle

\section{Introduction}

J. M. Dunn has described the structure of finite Sugihara monoids by showing in \cite{AB} that subdirectly irreducible Sugihara monoids are linearly ordered and by describing the structure of finite Sugihara chains in \cite{Dunn}.
We generalize Dunn’s results by dropping the idempotence and the finiteness axioms. 

In line with the tradition of representing classes of residuated lattices by simpler and better known structures such as groups or Boolean algebras, Mundici's celebrated categorical equivalence theorem represents MV-algebras -- a variety which corresponds to \L ukasiewicz logic $\mathbb L$ \cite{COM} -- by $\ell$-groups with strong units using a truncation construction \cite{Gamma}. 
By dropping the divisibility axiom of MV-algebras one obtains the variety of IMTL-algebras 
which correspond to the logic $\mathbb{IMTL}$ \cite{FSystems,Perfect,IMTLgamma,PriestleyIMTLis,IMTLstates,stIMTLeredeti,PerfectAgain,CCkripke,StIMTL,IMTLCIG,IMT3}.
The non-integral analogue of the class of IMTL-algebras shall be represented by $o$-groups in this paper.
By replacing the integrality axiom of IMTL-algebras
by one of its two natural non-integral analogues one obtains the class of odd and the class of even involutive semilinear FL$_e$-algebras, the former of which is also known as the class of IUL$^{fp}$-algebras corresponding to the 
Involutive Uninorm Logics with fixed point ($\mathbb{IUL}^{fp}$) introduced by G. Metcalfe in \cite{GeorgePhD}, see also \cite{MetMontSubstr}.
We shall represent all subdirectly irreducible members of these two classes by means of abelian $o$-groups. Subdirectly irreducible even or odd involutive 
 FL$_e$-algebras are linearly ordered.
Little is known about the structure of IMTL-algebras and an effective structural description for this class seems to be out of reach at present, arguably, such a structural description is impossible.
However, as shown in this paper, the non-integral counterpart of the class of IMTL-algebras possesses a representation by groups, similar to MV-algebras or to cancellative commutative residuated lattices \cite{MoTs}.

Whether Involutive Uninorm Logic ($\mathbb{IUL}$) is standard complete or not is a problem posed by G.~Metcalfe and F.~Montagna \cite{MetMontSubstr} and has become a long-standing open problem which has invoked considerable effort to solve it with no avail\footnote{Recently it is claimed to be proved in \cite{IUL?} by using proof-theoretic methods. However, the scientific community has doubts about the correctness of the proof, see the remark of the author himself in \cite[second section in page~43]{IUL?}.}. 
The representation theorem in this paper may serve as a step toward settling the problem of standard completeness of $\mathbb{IUL}$ in an algebraic manner. 

Group theory is the science of symmetries whereas the theory of inverse semigroups is the science of partial symmetries \cite{InverseSemigroups}.
Our main result is reminiscent to the representation theorem of a subclass of inverse semigroups, called Clifford semigroups. 
Every Clifford semigroup is isomorphic to a strong semilattice of groups \cite[Theorem 12 in Section 5.2]{InverseSemigroups} (see also \cite{Saito}), and the starting point of the strong semilattice construction is a direct system of groups indexed by a semilattice.
In our case the starting point is a direct system of abelian $o$-groups indexed by a chain and equipped with some extra structure. The definition of the {\em product} in the strong semilattice construction is of the same P\l{}onka sum fashion (see \cite{plonka}) as in our construction in (\ref{EgySzeruTe}). However, in our case we need to modify the abelian $o$-groups prior to applying (\ref{EgySzeruTe}), and we also need to handle the ordering, the residual operation, the isotonicity of the product with respect to the ordering etc.
In particular, we introduce in this paper an ordering on P\l{}onka sums, called the directed lexicographic order.

P\l{}onka sums have found applications not only in computer science, in particular in the theory of program semantics \cite{PlonkaProgram}, but
recently its corresponding regular varieties have been unexpectedly connected to logic.
They provide algebraic semantics to logics obtained by imposing a deductive filter to other logics; for instance, Paraconsistent Weak Kleene logic coincides with the regularization of the variety of Boolean algebras \cite{BonzioPlonka,BMB}.

Integral residuated lattices has been widely studied in the literature \cite{AglMon,BCGJT,BlokRaftery,COM,GenMV,hajekBOOK,JKpsBCK}.
However, as noted by N. Galatos and J. G. Raftery in \cite{GalRaf}, 
non-integral residuated structures and consequently, substructural logics without the weakening rule are less understood than their integral counterparts. To overcome this,
they established category equivalences to carry over the algebraic knowledge on integral structures to non-integral ones \cite{GalRaf,Gal2015}. These kind of categorical approaches 
use some algebraic background knowledge which is available if idempotency is postulated. Our paper makes a step in the direction of providing an algebraic insight into the non-idempotent, non-integral, non-divisible setting.

Prominent examples of odd involutive FL$_e$-algebras are abelian $\ell$-groups \cite{AF} and odd Sugihara monoids.
These two classes of algebras represent two extremities: there is a single idempotent element in any lattice-ordered abelian group, whereas all elements are idempotent in any odd Sugihara monoid. 
The former class constitutes an algebraic semantics of Abelian Logic \cite{casari,MeyerAbelian,AbelianNow} while the latter constitutes an algebraic semantics of $\mathbb{IUML}^*$, which is a logic at the intersection of relevance logic and many-valued logic \cite{GalRaf}.
Our representation theorem puts these two rather distant logics (and more) under the same methodological umbrella. 

Despite the extensive literature devoted to classes of residuated lattices, there are still few 
effective structural descriptions.
All these results in \cite{AglMon,COM,GenMV,hajekBOOK,JF,JKpsBCK,lawson,Mos57,GJM,IdemP,Olson} (and also others where the focus is not a structural description in the first place such as in \cite{GalRaf,Gal2015} for example)
postulate semilinearity which renders the subdirectly irreducible members linearly ordered, and either integrality together with the naturally ordered condition\footnote{Or its dual notion, called divisibility.} or idempotency.
Our study contributes to the structural description of residuated lattices which are semilinear but neither integral nor naturally ordered nor idempotent; a first such effective structural description to the best of our knowledge besides \cite{JS_Hahn,Jenei_Hahn_err}. 

A representation theorem has been presented in \cite{JS_Hahn,Jenei_Hahn_err} for those odd involutive FL$_e$-chains where the number of idempotent elements of the algebra is finite by means of partial sublex products of abelian $o$-groups, which are well understood mathematical objects that are much more regular than what had been expected to need for describing these particular FL$_e$-chains.
In the present paper we prove a representation theorem for both even and odd involutive FL$_e$-chains without assuming any constrains on the set of their idempotent elements. 
The conic case is also settled.
While the construction in the representation of \cite{JS_Hahn,Jenei_Hahn_err} is done by starting with an abelian $o$-group and iteratively enlarging it by other abelian $o$-groups until the obtained structure becomes isomorphic to the given algebra, here we present a structural description using direct systems, without referring to iteration.
To this end, a core auxiliary result concerns a one-to-one correspondence between two lattice ordered classes: odd involutive FL$_e$-algebras and even involutive FL$_e$-algebras with an idempotent falsum constant. Here we make use of and further develop ideas of E. Casari \cite{casari}, see Section \ref{WaRMuPEVEN_ID}.

Since $\ell$-groups are very specific mathematical objects compared to residuated lattices, having a representation of a class of residuated lattices by subclasses of $\ell$-groups is a rarity.
It is reasonable to expect that the inherent larger complexity of the more general classes of residuated lattices renders such a representation, if exists, quite involved in its constructional part.
This phenomenon can be observed in \cite{Jenei_Hahn_err,JS_Hahn}, for example, where the notions of two different partial sublex product constructions together with the technique of iteration have been used to describe the class of residuated lattices which is in our focus in the present paper with the additional assumption that the number of idempotent elements of the algebra is finite.
In the present paper the complexity of the residuated structure is coded in the system of homomorphisms and in the delicate way of constructing the algebra from the direct system.
Applications of the main result of this paper to amalgamation and densification problems are foreshadowed in \cite{JS_amalg_dens}.

\section{Preliminaries}
An {\em FL$_e$-algebra} 
is a structure $( X, \wedge,\vee, \teALONE, \ite{\te}, t, f )$ such that 
$(X, \wedge,\vee )$ is a lattice, $( X,\leq, \teALONE,t)$ is a commutative, 
residuated monoid (the unit element $t$ is also referred to as the {\em truth} constant), and $f$ is an arbitrary constant, called the {\em falsum} constant. Being residuated means that there exists a binary operation $\ite{\te}$,
called the residual operation of $\teALONE$, such that $\g{x}{y}\leq z$ if and only if $\res{\te}{x}{z}\geq y$. This equivalence is called adjointness condition, ($\teALONE,\ite{\te}$) is called an adjoint pair. Equivalently, for any $x,z$, the set $\{v\ | \ \g{x}{v}\leq z\}$ has its greatest element, and $\res{\te}{x}{z}$, the residuum of $x$ and $z$, is defined as this element: $\res{\te}{x}{z}:=\max\{v\ | \ \g{x}{v}\leq z\}$; this is called the residuation condition. Being residuated implies that $\teALONE$ is increasing.
One defines the {\em residual complement operation} by $\nega{x}=\res{\te}{x}{f}$ and calls an FL$_e$-algebra {\em involutive} if $\nega{(\nega{x})}=x$ holds.
An involutive FL$_e$-algebra is called {\em odd} if the residual complement operation leaves the unit element fixed, that is, $\nega{t}=t$, and {\em even} if the following (two) quasi-identities hold: $x<t$ $\Leftrightarrow$ $x\leq f$. 
The former condition is equivalent to $f=t$, while 
the latter quasi-identities are equivalent to assuming that 
$f$ 
is the lower cover of 
$t$ (and $t$ 
is the upper cover of 
$f$)
if chains (or more generally conic algebras) are considered. 
An FL$_e$-algebra is called {\em integral} if the unit element of the multiplication is the top element of its universe.
Commutative residuated lattices are the $f$-free reducts of FL$_e$-algebras. 
The geometric meaning of $\nega{t}=f$, which is valid in all involutive FL$_e$-algebras, is that the two constants are
positioned symmetrically inside the underlying set. 
Therefore, one extremal setting is the integral case, when $t$ and $f$ are the top and bottom elements of $X$, respectively, and the other extremal one is the even or odd case when the two constants are both \lq\lq in the middle of $X$\rq\rq.
A bottom-top setup is not possible: if a residuated lattice has a bottom element then that element is annihilating, hence cannot be the unit of the multiplication\footnote{Unless the algebra is trivial.}.
Both odd or even involutive FL$_e$-chains and IMTL-chains are involutive residuated chains with a single additional postulate: the integrality condition of IMTL-chains postulates the unit element 
to be in one of its possible extremal positions (top), whereas the odd or even condition postulates the unit element to be in its other extremal position (in the middle). 

Section~\ref{PReLimiNAriEs} contains results on the local unit element function of involutive FL$_e$-algebras, a key concept for our representation theorem.
Even involutive FL$_e$-algebras with non-idempotent and idempotent falsum constants will be characterized with respect to odd involutive FL$_e$-algebras in Sections~\ref{WaRMuPEVEN_NID} and \ref{WaRMuPEVEN_ID}, respectively.
In Section~\ref{COnsTRUctioN2} even and odd involutive FL$_e$-chains will be partitioned into their so-called layer algebras, which are also even or odd involutive FL$_e$-chains but more specific ones than the original algebra in the sense that they are either cancellative or are close to being such: in the latter case there exists a canonical homomorphism which maps the layer algebra into a cancellative one from which the layer algebra can be uniquely recovered.
This specificity allows for establishing a connection between them and abelian $o$-groups in Section~\ref{COnsTRUctioN}, by using the characterizations of Sections~\ref{WaRMuPEVEN_NID} and \ref{WaRMuPEVEN_ID}. 
These lead to the main result of the paper in Section~\ref{NaVegre}: a one-to-one correspondence in a constructive manner between the class of all even or odd involutive commutative residuated chains and the class of bunches of layer groups. Layer groups are direct systems of abelian $o$-groups equipped with further structure.

\section{The local unit element function}\label{PReLimiNAriEs}

Let $(X, \leq)$ be a poset.
For $x,y\in X$ we say that $y$ is a cover of $x$ if $y>x$ and there exists no $z\in X$ such that $y>z>x$.
For $x\in X$ let $x_\uparrow$ be the unique cover of $x$ if such exists, and let $x_\uparrow=x$ otherwise.
Define $x_\downarrow$ dually.
Call ${}_\downarrow$ and ${}_\uparrow$ the neighbour operations of $(X, \leq)$.
A partially ordered algebra with a poset reduct is called {\em discretely ordered} if for any element $x$, $x_\downarrow<x<x_\uparrow$ holds.
If $\komp$ is an order-reversing involution of $X$ then it holds true that 
\begin{equation}\label{FelNeg_NegLe}
\nega{x}_\uparrow=\nega{(x_\downarrow)} \ \ \   \mbox{ and }  \ \ \ \nega{x}_\downarrow=\nega{(x_\uparrow)} .
\end{equation}

Algebras will be denoted by bold capital letters, their underlying sets by the same regular letter unless otherwise stated.
Sometimes the lattice operators of an FL$_e$-algebra will be replaced by their induced ordering $\leq$ in the signature, in particular, if an FL$_e$-{\em chain} is considered, that is, if the ordering is total.
Call the elements $x\geq t$ of an FL$_e$-algebra $\mathbf X=(X, \wedge, \vee,\teALONE, \ite{\te}, t, f)$ {\em positive}.
Call $\mathbf X$ {\em conic} if all elements of $X$ are comparable with $t$.
Assume $\mathbf X$ is involutive. 
For $x\in X$ let $$\tau(x)=\res{\te}{x}{x},$$ or equivalently, define $\tau(x)$ to be the greatest element of $Stab_x=\{ z\in X \ | \ \g{z}{x}=x \}.$
A key step toward our representation theorem is to understand the role of the $\tau$ function.
In investigating more specific odd involutive FL$_e$-algebras $\tau(x)$ was recognized to play the role of the \lq absolute value of $x$\rq\ and was denoted by $|x|$, see \cite{IdemP}.
The definition of layers in (\ref{RetegDefinicio}) and Lemma~\ref{tuttiINVOLUTIVE} reveal the true nature of $\tau$ in the present, more general setting:
$\tau$ can be view as a \lq local unit element\rq\ function.
For any positive idempotent element $u$ define
\begin{equation}\label{RetegDefinicio}
X_u=\{x\in X : \tau(x)=u\}
\end{equation}
and call it the $u$-layer of $\mathbf X$.
The $X_u$'s form a partition of $X$ by 
claim~\ref{zetaRange_idempotent} below, and 
if $x\in X_u$ then $\tau(x)$ is the unit element for the subset $X_u$ of $X$ by claim~\ref{tnelnagyobb}\footnote{For a first glance it might occur to the reader that the definition of $X_u$ is equivalent to the well-known way of localizing in semigroup theory, namely, for an idempotent $u$, to the universe
$M_u=\{x\in X : xu=x\}$ of the greatest unitary subsemigroup of $\mathbf X$ with unit $u$. 
It is not the case if $u>t$.
Rather, $X_u$ can be regarded as the \lq\lq outer edge\rq\rq\ of $M_u$ in the following sense:
it holds true that
$M_u=\bigcup_{v\leq u}X_v$
and
$X_u=M_u\setminus \bigcup_{v<u}X_v$
($u$ and $v$ are positive idempotents).
}.
Corresponding to the rotation-annihilation construction the idea of {\em decomposing} a class of FL$_e$-algebras with the help of a closely related notion, the {\em skeleton} function of the algebra\footnote{The skeleton function was defined there as the residual complement of what we coined local-unit function.} has been presented already  in \cite{JSIII}.
The name local unit function and the corresponding idea of using it to {\em localize} elements of the algebra has been introduced in \cite{JS_Hahn}, and the way of localizing as done in (\ref{RetegDefinicio}) is being introduced in the present paper, see also \cite{GroupReprArXiv}.
Some of the statements of the following lemma can be found in \cite{JS_Hahn} or in \cite{gjko}, too. We include their proofs to keep the paper self-contained.
As usual, when multiplication is denoted by $\teALONE$, we write $xy$ instead of $x\teALONE y$.

\begin{lemma}\label{tuttiINVOLUTIVE}
Let $\mathbf X=(X, \wedge, \vee,\teALONE, \ite{\te}, t, f)$ be an involutive FL$_e$-algebra.
The following statements hold true.
\begin{enumerate}[(i)]
\item\label{MiNdig}
$\nega{t}=f$,
\item\label{eq_quasi_inverse}
$\res{\te}{x}{y} = \nega{\left(\g{x}{\nega{y}}\right)}$,
\item\label{eq_feltukrozes}
if $t\geq f$ then $\g{x}{y}\leq\nega{\left(\g{\nega{x}}{\nega{y}}\right)}$,
\item\label{eq_feltukrozes_CS}
if the algebra is conic and either even or $t\leq f$ then $y_1>y$ implies $\nega{(\g{\nega{x}}{\nega{y}})}\leq\g{x}{y_1}$,
\item\label{tnelnagyobb}
$\g{\tau(x)}{x}=x$ and $\tau(x)\geq t$,
\item\label{ZetaOfIdempotent} 
$u\geq t$ is idempotent if and only if $\tau(u)=u$,
\item\label{zetaRange_idempotent}
$\{ \tau(x): x\in X \}$ is equal to the set of positive idempotent elements of $\mathbf X$,
\item\label{diagonalCOSinverse}
$\g{x_1}{y}>\g{x}{y}$ holds whenever $x_1>x$ and $y$ is invertible,
\item\label{NegArAAA}
$\tau(x)=\tau(\nega{x})$,
\item\label{canc_Fremegy}
if $\teALONE$ is cancellative then for $x\in X$, $\g{x}{\nega{x}}=f$,
\item\label{csoportLesz}
if $\mathbf X$ is odd then the $(X, \wedge, \vee,\teALONE, t)$-reduct of $\mathbf X$  is a {lattice-ordered abelian group} if and only if $\teALONE$ is cancellative \footnote{When we (loosely) speak about a subgroup of odd involutive FL$_e$-algebra $\mathbf X$ in the sequel, we shall mean a cancellative subalgebra of $\mathbf X$.},
\item\label{item_boundary_zeta}
for $x\geq t$, $\tau(x)\leq x$ holds,
\item\label{def A referred}
if $\mathbf X$ is an odd chain then $X_t$ contains all the invertible elements of $X$, and it is a universe of an abelian $o$-group, its inverse operation is the residual complement operation.
\item\label{eGGyelLejjebB}
if $\mathbf X$ is an odd or even, cancellative, discretely ordered chain then $\g{z}{t_\downarrow}=z_\downarrow$,

\item\label{diagonalSZIGORU}
if $\mathbf X$ is an odd or even chain then 
$\g{x}{y}<\g{x_1}{y_1}$ holds whenever $x<x_1$ and $y<y_1$,
\item\label{tau_lemma}
if $\mathbf X$ is an odd or even chain and $A$ is an $\mathbf X$-term which contains only the operations $\teALONE$, $\ite{\te}$ and \, $\komp$
then
for any evaluation $e$ of the variables of $A$ into $X$, $\tau(e(A))$ equals the maximum of the $\tau$-values of the variables and constants of $A$ under $e$,
\end{enumerate}
\end{lemma}
\begin{proof}
\begin{enumerate}[(i)]
\item
Obvious by residuation.
\item
Using that $\komp$ is an involution one obtains 
$\nega{(\g{x}{\nega{y}})}=\res{\te}{\left(\g{x}{\nega{y}}\right)}{f}=\res{\te}{x}{\left( \res{\te}{\nega{y}}{f} \right)}=\res{\te}{x}{y}$ (folklore). 
\item
Next,
$\g{(\g{x}{y})}{(\g{\nega{x}}{\nega{y}})}=
\g{[\g{x}{(\res{\te}{x}{f})}]}{[\g{y}{(\res{\te}{y}{f})}]}
\leq \g{f}{f}
\leq \g{t}{f}=f$,
hence $\g{x}{y} \leq \res{\te}{(\g{\nega{x}}{\nega{y}})}{f}$ follows by adjointness. 
\item
Since  the algebra is conic and involutive, every element is comparable with $f$, too.
Indeed, if for any $a\in X$, $a$ were not comparable with $f$ then, since $\komp$ is an order reversing involution, $\nega{a}$ were not comparable with $\nega{f}$, and $\nega{f}=t$ since the algebra is involutive, a contradiction.
Therefore, by residuation, $y_1>\nega{(\nega{y})}=\res{\te}{\nega{y}}{f}$ implies 
$\g{y_1}{\nega{y}}\not\leq f$, that is, 
$\g{y_1}{\nega{y}}>f$. 
If $t\leq f$ then $\g{y_1}{\nega{y}}\geq t$ follows.
Likewise, if the algebra is even then $\g{y_1}{\nega{y}}\not\leq f$
implies $\g{y_1}{\nega{y}}\not< t$, that is, 
$\g{y_1}{\nega{y}}\geq t$, since the algebra is conic.
Therefore, 
$\nega{(\g{x}{y_1})}=\g{\nega{(\g{x}{y_1})}}{t}\leq\g{\nega{(\g{x}{y_1})}}{\g{y_1}{\nega{y}}}
\overset{\claim \ref{eq_quasi_inverse}}{=}
\g{(\g{y_1}{(\res{\te}{y_1}{\nega{x}})})}{\nega{y}}\leq\g{\nega{x}}{\nega{y}}$ follows.
\item
Since $t$ is the unit element, $Stab_x$ is nonempty. Therefore, by residuation $\tau(x)=\res{\te}{x}{x}$ is its largest element, hence $\g{\tau(x)}{x}=x$ holds.
Since $\g{t}{x}=x$, $\tau(x)\geq t$ follows by residuation.
\item
If $u\geq t$ is idempotent then from $\g{u}{u}=u$, $\res{\te}{u}{u}\geq u$ follows by adjointness. But for any $z>u$, $\g{u}{z}\geq\g{t}{z}=z>u$, 
hence $\tau(u)=u$ follows. 
On the other hand, by claim~\ref{tnelnagyobb}, $\tau(u)=u$ implies $u\geq t$, and also the idempotency of $u$ follows since $\g{u}{u}=\g{u}{\tau(u)}=u$.
\item
If $u>t$ is idempotent then claim~\ref{ZetaOfIdempotent} shows that $u$ is in the range of $\tau$.
If $u$ is in the range of $\tau$, that is $\tau(x)=u$ for some $x\in X$ then
if $\tau(\tau(x))=\tau(x)$ then it 
implies $\tau(u)=u$, hence $u$ is a positive idempotent element by claim~\ref{ZetaOfIdempotent}, and we are done.
Hence it suffices to prove $\tau(\tau(x))=\tau(x)$ for all $x$.
By claim~\ref{eq_quasi_inverse}, $\res{\te}{x}{x}=\tau(x)$ is equivalent to $\g{x}{\nega{x}}=\nega{\tau(x)}$.
Hence, $\g{\tau(x)}{\nega{\tau(x)}}=\g{\tau(x)}{(\g{x}{\nega{x}})}=\g{\g{(\tau(x)}{x})}{\nega{x}}
\overset{\claim \ref{tnelnagyobb}}{=}
\g{x}{\nega{x}}=\nega{\tau(x)}$ follows, which is equivalent to $\tau(\tau(x))=\res{\te}{\tau(x)}{\tau(x)}=\tau(x)$.
\item
$\g{x_1}{y}\geq\g{x}{y}$ holds by monotonicity, and if $\g{x_1}{y}=\g{x}{y}$ then $x_1=\g{\g{x_1}{y}}{y^{-1}}=\g{\g{x}{y}}{y^{-1}}=x$, a contradiction.
\item
By claim~\ref{eq_quasi_inverse} and the involutivity of $\komp$, 
$\tau(x)=\res{\te}{x}{x}=\nega{(\g{x}{\nega{x}})}=\nega{(\g{\nega{x}}{\nega{\nega{x}}})}=\res{\te}{\nega{x}}{\nega{x}}=\tau(\nega{x})$.
\item
Since $\teALONE$ is cancellative, the strictly increasing nature of $\teALONE$ clearly follows: if $u<v$ then $\g{u}{w}<\g{v}{w}$. Therefore, $\g{x}{a}>\g{x}{t}=x$ for any $a>t$, and hence $\res{\te}{x}{x}=t$. 
An application of claim~\ref{eq_quasi_inverse} ends the proof.
\item
Necessity is straightforward, sufficiency follows from claim~\ref{canc_Fremegy} since $f=t$.
\item
For $x\geq t$, 
$x\overset{\claim \ref{tnelnagyobb}}{=}
\g{x}{\tau(x)}\geq\g{t}{\tau(x)}=\tau(x)$.
\item
If $x\in X_t$ then $\tau(x)=\res{\te}{x}{x}=t$ and hence $\g{x}{\nega{x}}
\overset{\claim \ref{eq_quasi_inverse}}{=}
\nega{t}
\overset{\claim \ref{MiNdig}}{=}
f=t$ holds since the algebra is odd. This shows that all elements of $X_t$ are invertible in $X$ and the inverse operation is $\komp$.
On the other hand, if $x\in X$ is invertible, that is, $\g{x}{y}=t$ for some $y\in X$ then $\tau(x)=t$: for every $z>t$ it holds true that $\g{x}{z}>x$ since the opposite, that is, $\g{x}{z}=x$ would imply $z=\g{t}{z}=\g{(\g{y}{x})}{z}=\g{y}{(\g{x}{z})}=\g{y}{x}=t$, a contradiction.
Invertible element are clearly closed under $\teALONE$ and $\komp$.
Since the order is total, the meet and the join of two invertible elements are also invertible.
\item
Since $\teALONE$ is cancellative, $\teALONE$ is strictly increasing; we shall use it without further mention.
Now, $\g{z}{t_\downarrow}<\g{z}{t}=z$ holds since $t_\downarrow<t$ follows from the algebra being discretely ordered. 
Contrary to the statement, assume that there exists $a$ such that $\g{z}{t_\downarrow}<a<z.$
Multiplying with $\nega{z}$,
$
\g{f}{t_\downarrow}<\g{a}{\nega{z}}<f
$
follows by claim~\ref{canc_Fremegy}.
In the odd case it yields $t_\downarrow<\g{a}{\nega{z}}<t$, a contradiction to the definition of ${}_\downarrow$. 
In the even case it yields
$\g{t_\downarrow}{t_\downarrow}<\g{a}{\nega{z}}<t_\downarrow.$
Since $t_\downarrow=\g{t}{t_\downarrow}<\g{t_\uparrow}{t_\downarrow}<\g{t_\uparrow}{t}=t_\uparrow$, we obtain $\g{t_\uparrow}{t_\downarrow}=t$, and hence multiplication by $t_\uparrow$ implies
$t_\downarrow<\g{t_\uparrow}{\g{a}{\nega{z}}}<t,$
a contradiction to the definition of ${}_\downarrow$.

\item
By claim~\ref{eq_feltukrozes}, 
$\nega{\left(\g{\nega{x}}{\nega{y}}\right)}\geq\g{x}{y}$ holds, hence it suffices to prove $\g{x_1}{y_1}>\nega{\left(\g{\nega{x}}{\nega{y}}\right)}$.
Assume the opposite, which is $\g{x_1}{y_1}\leq\nega{\left(\g{\nega{x}}{\nega{y}}\right)}$ since $(X,\leq)$ is a chain.
By adjointness we obtain $\g{(\g{\nega{x}}{x_1})}{(\g{\nega{y}}{y_1})}=\g{(\g{x_1}{y_1})}{(\g{\nega{x}}{\nega{y}})}\leq f$, and
from $x_1>x=\nega{(\nega{x})}$, $\g{\nega{x}}{x_1}>f$ follows by residuation since $(X,\leq)$ is a chain.
\\
In the odd case these reduce to $\g{(\g{\nega{x}}{x_1})}{(\g{\nega{y}}{y_1})}\leq t$ and $\g{\nega{x}}{x_1}>t$.
Analogously we obtain $\g{\nega{y}}{y_1}>t$.
Therefore $\g{(\g{\nega{x}}{x_1})}{(\g{\nega{y}}{y_1})}\geq \g{(\g{\nega{x}}{x_1})}{t}=\g{\nega{x}}{x_1}>t$ follows, a contradiction.
\\
In the even case these reduce to $\g{(\g{\nega{x}}{x_1})}{(\g{\nega{y}}{y_1})}<t$ and $\g{\nega{x}}{x_1}\geq t$.
Analogously we obtain $\g{\nega{y}}{y_1}\geq t$.
Therefore $\g{(\g{\nega{x}}{x_1})}{(\g{\nega{y}}{y_1})}\geq \g{t}{t}=t$ follows, a contradiction.

\item
We have already seen $\tau(x)=\tau(\nega{x})$ in claim~\ref{NegArAAA}.\\
Next, we claim $\tau(\g{x}{y})=\max(\tau(x),\tau(y)) (=\g{\tau(x)}{\tau(y)}$\footnote{For positive idempotents $u\leq v$ it holds true that $\g{u}{v}=\max(u,v)$
since $v=\g{t}{v}\leq\g{u}{v}\leq\g{v}{v}=v$.}).
Indeed, $\tau(\g{x}{y})\geq\tau(x)$ holds by residuation since $\res{\te}{\g{x}{y}}{\g{x}{y}}\geq\res{\te}{x}{x}$ is equivalent to $\g{\g{y}{x}}{(\res{\te}{x}{x})}\leq\g{x}{y}$. 
Assume $z:=\tau(\g{x}{y})>\max(\tau(x),\tau(y))$.
Since $\tau$ assigns to $x$ the greatest element of the stabilizer set of $x$, 
therefore
$z>\max(\tau(x),\tau(y))$ implies that $z$ does not stabilize $x$ neither $y$, hence
$x<\g{z}{x}$ and $y<\g{z}{y}$ holds by the monotonicity of $\teALONE$.
On the other hand, 
$\g{(\g{z}{x})}{(\g{z}{y})}=\g{(\g{(\g{x}{y})}{z})}{z}\overset{\claim \ref{tnelnagyobb}}{=}
\g{x}{y}$ follows, a contradiction to claim~\ref{diagonalSZIGORU}. This settles the claim.
\\
By claim~\ref{eq_quasi_inverse}, any term which contains only the connectives $\teALONE$, $\ite{\te}$ and \, $\komp$
can be represented by an equivalent term using the same variables and constants but containing only $\teALONE$ and $\komp$. 
An easy induction on the recursive structure of this equivalent term using the two claims above concludes the proof.
\end{enumerate}
\end{proof}

The next lemma states that for an odd involutive FL$_e$-algebra, it is exactly cancellativity which is needed to make it a lattice ordered group.

\begin{lemma}\label{inducedDEF}\rm
\begin{enumerate}
\item
For a cancellative odd involutive FL$_e$-algebra 
$$\mathbf X=(X,\wedge, \vee,\teALONE,\ite{\te},t,t)$$
with residual complement $\komp$,
$\lambda(\mathbf X)=(X,\wedge, \vee,\teALONE,\ { }^{-1},t)$
is a lattice-ordered abelian group,
called the {\em lattice-ordered abelian group induced by $\mathbf X$},
where
\begin{equation}\label{EzAzInverz}
x^{-1}=\nega{x}=\res{}{x}{t}
,
\end{equation}
\item\label{CsopToAlg}
For a lattice-ordered abelian group $$\textbf{\textit{G}}=(G,\leq,\teALONE,\ { }^{-1},t),$$
$\iota(\textbf{\textit{G}})=(G,\wedge, \vee,\teALONE,\ite{\te},t,t)$
is a cancellative odd involutive FL$_e$-algebra,
called the {\em cancellative odd involutive FL$_e$-algebra induced by $\textbf{\textit{G}}$},
where 
\begin{equation}\label{EzAResiduum}
\res{}{x}{y}=\g{x^{-1}}{y},
\end{equation}
\begin{equation}\label{EzAResComplement}
\nega{x}=
x^{-1},
\end{equation}
\item 
With the above assumptions it holds true that
$\iota(\lambda(\mathbf X))=\mathbf X$ and $\lambda(\iota(\textbf{\textit{G}}))=\textbf{\textit{G}}$.
\end{enumerate}
\end{lemma}
\begin{proof}
Claim~\ref{csoportLesz} in Lemma~\ref{tuttiINVOLUTIVE} confirms the first statement, the rest is folklore or obvious.
\end{proof}

The following lemma will simplify the proof of Theorem~\ref{sPliT} and Lemma~\ref{BUNCHalg_X}.
Let ${\mathcal M}=(M,\leq,\teALONE)$ be a structure such that $(M,\leq)$ is a poset and $(M,\teALONE)$ is a commutative semigroup.
Call $c\in M$ a dualizing element\footnote{Dualizing elements have been defined only in {\em residuated} structures in the literature, see e.g. \cite[Section 3.4.17.]{gjko}.}  of $\mathcal M$, if 
(i) for $x\in M$ there exists $\res{\te}{x}{c}$\footnote{That is, the exists the greatest element of the set $\{ z\in M \ | \ \g{x}{z\leq c}\}$.}, and 
(ii) for $x\in M$, $\res{\te}{(\res{\te}{x}{c})}{c}=x$.

\begin{lemma}\label{SimpLER}
If there exists a dualizing element $c$ of $\mathcal M$ then $\teALONE$ is residuated and its residual operation is given by $\res{\te}{x}{y}=\res{\te}{\left(\g{x}{(\res{\te}{y}{c})}\right)}{c}$.
\end{lemma}
\begin{proof}
$\g{z}{x}\leq y$ is equivalent to $\g{z}{x}\leq\res{\te}{\left(\res{\te}{y}{c}\right)}{c}$.
By adjointness it is equivalent to $\g{\left(\g{z}{x}\right)}{(\res{\te}{y}{c})}\leq c$.
By associativity it is equivalent to $\g{z}{\left(\g{x}{(\res{\te}{y}{c})}\right)}\leq c$, which is equivalent to $z\leq \res{\te}{\left(\g{x}{(\res{\te}{y}{c})}\right)}{c}$ by adjointness.
By residuation $\res{\te}{x}{y}=\res{\te}{\left(\g{x}{(\res{\te}{y}{c})}\right)}{c}$ follows.
\end{proof}

\section{Odd and even involutive FL$_e$-chains vs.\,bunches of layer algebras}\label{COnsTRUctioN2}

We shall prove the main theorem of the paper for three different kinds of involutive FL$_e$-chains:
for odd involutive FL$_e$-chains,
for even involutive FL$_e$-chains with an idempotent falsum, and
for even involutive FL$_e$-chains with a non-idempotent falsum.

\begin{definition}\label{DEFbunch2}
Let $(\kappa,\leq_\kappa)$ be a totally ordered set with least element $t$, and
let an ordered triple $\langle \bar\kappa_I, \bar\kappa_J, \{t\}\rangle$
be a partition of $\kappa$, where $\bar\kappa_I$ and $\bar\kappa_J$ can also be empty. 
Define $\kappa_o$, $\kappa_J$, and $\kappa_I$ by one of the rows of Table~\ref{ThetaPsiOmega} \footnote{Explanation: if the first, the second, or the third row is used then the algebra corresponding to the bunch will be odd, even with a non-idempotent falsum constant, or even with an idempotent falsum constant, respectively.}. 
\begin{table}[ht]
\begin{center}
\caption{}
\begin{tabular}{c|c|cl}
 $\kappa_o$ & $\kappa_J$ & $\kappa_I$ \\
\hline 
\{t\} & $\bar\kappa_J$ & $\bar\kappa_I$ \\
\hline
$\emptyset$ & $\bar\kappa_J\cup\{t\}$ & $\bar\kappa_I$ \\
\hline
$\emptyset$ & $\bar\kappa_J$ & $\bar\kappa_I\cup\{t\}$ \\
\hline
\end{tabular}
\label{ThetaPsiOmega}
\end{center}
\end{table}

\noindent
Let
$\mathbf X_u=(X_u,\leq_u,,\teu,\ite{u},u,\negaM{u}{u})$
be a family of involutive FL$_e$-chains 
indexed by elements of $\kappa$ (let $\kompM{u}$ denote the residual complement operation, ${}_{\downarrow_u}$ and ${}_{\uparrow_u}$ the neighbour operations of $\mathbf X_u$),
such that $\mathbf X_u$ is
\begin{equation}\label{KiKiLesz}
\left\{
\begin{array}{ll}
\mbox{cancellative and odd} & \mbox{if $u\in\kappa_o$}\\
\mbox{discretely ordered, cancellative and even\blfootnote{Hence with a non-idempotent falsum.} }& \mbox{if $u\in\kappa_J$}\\
\mbox{even with an idempotent falsum satisfying $\gteu{x}{\negaM{u}{x}}=\negaM{u}{u}$} & \mbox{if $u\in\kappa_I$}\\
\end{array}
\right. ,
\setcounter{footnote}{10} 
\footnote{Hence with a non-idempotent falsum.}
\end{equation}
and such that for $u,v\in\kappa$, $u\leq_\kappa v$, there exist a
\begin{equation}\label{RHOhomo}
\mbox{homomorphism
$\rho_{u\to v}$
}
\end{equation}
 from the residuated lattice reduct of $\mathbf X_u$ 
 to the residuated lattice reduct of $\mathbf X_v$ satisfying 
\begin{itemize}
\item[(A1)] 
$\rho_{u\to u}=id_{X_u}$ and $\rho_{v\to w}\circ\rho_{u\to v}=\rho_{u\to w}$ 
\hfill (direct system property),
\item[(A2)] 
for $u<_\kappa v$,
$\rho_{u\to v}(u)=\rho_{u\to v}(\negaM{u}{u})
$
\footnote{If $u\in\kappa_o$ then $\rho_{u\to v}(u)=\rho_{u\to v}(\negaM{u}{u})
$
trivially holds since $\negaM{t}{t}=t$.}
\hfill (constants' collision condition).
\end{itemize}
Call 
${\mathcal A}=\langle \mathbf X_u, \rho_{u\to v} \rangle_{\langle \kappa_o, \kappa_J, \kappa_I,\leq_\kappa\rangle}$
a {\em bunch of layer algebras}.
Call the $\mathbf X_u$'s the layer algebras, call $\langle\kappa,\leq_\kappa\rangle$ the {\em skeleton}, call $\langle \kappa_o, \kappa_J, \kappa_I \rangle$ the {\em partition} of the skeleton, and call $\langle \mathbf X_u, \rho_{u\to v} \rangle_\kappa$ the direct system of layer algebras over $\kappa$.
Note that $\kappa$ can be recovered from its partition, (and ultimately, from $\mathcal A$) via $\kappa=\kappa_o\cup\kappa_J\cup\kappa_I$.
\end{definition}

\medskip
We prove that every odd or even involutive FL$_e$-chain can be represented by a unique bunch of layer algebras. Later, in Section~\ref{COnsTRUctioN} we prove that every bunch of layer algebras can be  represented by a unique bunch of layer groups.

\begin{lemma}\label{BUNCHalg_X}
The following statements hold true.
\begin{enumerate}
\item 
Given an odd or an even involutive FL$_e$-chain $\mathbf X=(X,\leq,\teALONE,\ite{\te},t,f)$ with residual complement operation $\komp$,
$${\mathcal A}_\mathbf X=\langle \mathbf X_u, \rho_{u\to v} \rangle_{\boldsymbol\kappa}$$
is a bunch of layer algebras, 
called the {\em bunch of layer algebras}\,of $\mathbf X$,
where
$\tau(x)=\res{\te}{x}{x}$, $\kappa=\{\tau(x) : x\in X\}$,
$\leq_\kappa\,=\,\leq\,\cap \ (\kappa\times \kappa)$,
$\bar\kappa_I=\{u\in \kappa\setminus\{t\} : \nega{u} \mbox{ is idempotent}\},$
$\bar\kappa_J=\{u\in \kappa\setminus\{t\} : \nega{u} \mbox{ is not idempotent}\}$,
$\kappa_o$, $\kappa_J$, $\kappa_I$ are defined by Table~\ref{MasodiKK},
\begin{table}[h]
\begin{center}
\caption{ }
\begin{tabular}{c|c|c|cll}
$\kappa_o$ & $\kappa_J$ & $\kappa_I$ & \\
\hline
\{t\} & $\bar\kappa_J$ & $\bar\kappa_I$ & if $\mathbf X$ is odd\\
\hline
$\emptyset$ & $\bar\kappa_J\cup\{t\}$ & $\bar\kappa_I$ & if $\mathbf X$ is even and $f$ is not idempotent \\
\hline
$\emptyset$ & $\bar\kappa_J$ & $\bar\kappa_I\cup\{t\}$ & if $\mathbf X$ is even and $f$ is idempotent\\
\hline
\end{tabular}
\label{MasodiKK}
\end{center}
\end{table}

$\boldsymbol\kappa=\langle \kappa_o, \kappa_J, \kappa_I,\leq_\kappa\rangle$,
for $u\in \kappa$,
\begin{equation}\label{EQezLeszAgUUUUU}
\mathbf X_u=(X_u,\leq_u,\teu,\ite{u},u,\nega{u}),
\end{equation}
where $$X_u=\{x\in X : \tau(x)=u\},$$
$\leq_u$, $\teu$, and $\ite{u}$ are restrictions of $\leq$, $\teALONE$, and $\ite{\te}$ to $X_u$,  
for $x\in X_u$, $\negaM{u}{x}=\res{\te}{x}{\nega{u}}$,
and for $u,v\in \kappa$, $u\leq_\kappa v$,  $\rho_{u\to v} : X_u \to X_v$ is given by
\begin{equation}\label{EzARho}
\rho_{u\to v}(x)=\g{v}{x}.
\end{equation}
\item\label{dfgdkjkjHKJhKJNHG}
Given a bunch of layer algebras
${\mathcal A}=\langle \mathbf X_u, \rho_{u\to v} \rangle_{\boldsymbol\kappa}$ with $\boldsymbol\kappa={\langle \kappa_o, \kappa_J, \kappa_I, \leq_\kappa\rangle}$,
$\mathbf X_u=(X_u,\leq_u,\teu,\ite{u},u,\negaM{u}{u})$,
and 
$\negaM{u}{x}=\res{u}{x}{\negaM{u}{u}}$, 
$$
\mathcal X_{\mathcal A}=(X,\leq,\teALONE,\ite{\te},t,\nega{t})
$$
is an involutive FL$_e$-chain,
called the {\em involutive FL$_e$-chain derived from $\mathcal A$},
where
\begin{equation}\label{EZazX}
X=\displaystyle\dot \bigcup_{u\in \kappa}X_u
,
\end{equation}
for $v\in\kappa$, $\rho_v : X\to X$ is defined by
\begin{equation}\label{P5}
\rho_v(x)=\left\{
\begin{array}{ll}
\rho_{u\to v}(x) & \mbox{ if $u<_\kappa v$ and $x\in X_u$}\\
x & \mbox{ if $u\geq_\kappa v$ and $x\in X_u$}
\end{array}
\right. ,
\end{equation}
for short,
for $x\in X_u$ and $y\in X_v$,\footnote{Note that for $u,v\in\kappa$, $uv=\max_\kappa(u,v)$ since $u,v$ are positive idempotents.}
\begin{equation}\label{RendeZes}
\mbox{$x<y$ iff $\rho_{uv}(x)<_{uv}\rho_{uv}(y)$ or $\rho_{uv}(x)=\rho_{uv}(y)$ and $u<_\kappa v$}
\footnote{
Alternatively, we may write 
$x\leq y$ iff 
$\rho_{uv}(x)\leq_{uv}\rho_{uv}(y)$ except if $u>_\kappa v$ and $\rho_{uv}(x)=\rho_{uv}(y)$
}
\end{equation}
\begin{equation}
\label{EgySzeruTe}
\g{x}{y}=\gteM{uv}{\rho_{uv}(x)}{\rho_{uv}(y)},
\end{equation}
\begin{equation}\label{IgyKellEztCsinalni}
\mbox{
$\nega{x}=\negaM{u}{x}$,
}
\end{equation}
\begin{equation}\label{IgYaReSi}
\res{\te}{x}{y}=\nega{(\g{x}{\nega{y}})},
\end{equation}
and $t$ is the least element of $\kappa$.
$\mathcal X_{\mathcal A}$ is odd if $t\in\kappa_o$, 
even with a non-idempotent falsum if $t\in\kappa_J$, and 
even with an idempotent falsum if $t\in\kappa_I$.
\item 
For a bunch of layer algebras $\mathcal A$, $\mathcal A_{\left(\mathcal X_{\mathcal A}\right)}=\mathcal A$, and
for an odd or even involutive FL$_e$-chain $\mathbf X$, $\mathcal X_{\mathcal A_\mathbf X}=\mathbf X$.
\end{enumerate}
\end{lemma}
\begin{proof}
(1):
$\kappa$ is the set of positive idempotent elements of $\mathbf X$ by claim~\ref{zetaRange_idempotent} in Lemma~\ref{tuttiINVOLUTIVE}.
Therefore, the least element of $\kappa$ is $t$, and 
$\kappa$, being a subset of $X$, is totally ordered.
The ordered triple $\langle \bar\kappa_I, \bar\kappa_J, \{t\}\rangle$ is clearly a partition of $\kappa$, where $\bar\kappa_I$ and $\bar\kappa_J$ can also be empty. 
\\
Let $u\in\kappa$.
$X_u$ is nonempty since $u\in X_u$  
holds by claim~\ref{ZetaOfIdempotent} in Lemma~\ref{tuttiINVOLUTIVE}, 
and, being a subset of $X$, $X_u$ is totally ordered by $\leq_u$.
$X_u$ is closed under $\teu$, $\ite{u}$, and $\komp$ by claim~\ref{tau_lemma} in Lemma~\ref{tuttiINVOLUTIVE}, and thus $\negaM{u}{u}=\res{\te}{u}{\nega{u}}\in X_u$.
Since $\tau(x)=u$ holds for $x\in X_u$, therefore
$
\g{x}{u}=\g{x}{\tau(x)}
\overset{L\ref{tuttiINVOLUTIVE}\ref{tnelnagyobb}}{=}
x
$ shows that $u$ is the unit element of $\mathbf X_u$.
For $x\in X_u$, 
\begin{equation}\label{MeGegyeznekPRIME}
\mbox{
$\negaM{u}{x}=\nega{x}$
}
\end{equation}
holds since
$
\negaM{u}{x}
=
\res{\te}{x}{\nega{u}}
=
\res{\te}{x}{(\res{\te}{u}{f})}
\overset{\teALONE\ is\ residuated}{=}
\res{\te}{(\g{x}{u})}{f}
=
\res{\te}{x}{f}
=\nega{x}
$.
Therefore,
$
\negaM{u}{u}
=\nega{u}
$,
and hence
$
\negaM{u}{x}
=
\res{\te}{x}{\negaM{u}{u}}
$.
Summing up, 
\begin{equation}\label{ValoBanAZ}
\mbox{
$\mathbf X_u=(X_u,\leq_u,\teu,\ite{u},u,\nega{u})$ is an involutive FL$_e$-chain.
}
\end{equation}
Next we prove that the $\mathbf X_u$'s satisfy (\ref{KiKiLesz}).
\begin{description}
\item
If $u\in\kappa_o$ then $\mathbf X$ is odd by Table~\ref{MasodiKK}.
By claim~\ref{def A referred} in Lemma~\ref{tuttiINVOLUTIVE} and by Lemma~\ref{inducedDEF}, 
$\mathbf X_u$
is a cancellative odd involutive FL$_e$-chain. 
\item
If $u\in\kappa_J$ then by Table~\ref{MasodiKK},
$u=t$ or $u\in\bar\kappa_J$.
In both cases $\nega{u}$ is not idempotent.

We prove that $X_u$ is discretely ordered by showing
\begin{equation}\label{SZORuVESSZO2}
\g{x}{\nega{u}}=x_{\downarrow_u}<x
\end{equation}
for $x\in X_u$, where ${}_{\downarrow_u}$ denotes the neighbour operation on $X_u$.
It holds true that
\begin{equation}\label{JHGkcllvVjJKlc}
\nega{u}<t.
\end{equation}
Indeed, if $u\in\bar\kappa_J$ then $u>t$ and the involutivity of $\komp$ on $X$ implies $\nega{u}<\nega{t}
\overset{L\ref{tuttiINVOLUTIVE}\ref{MiNdig}}{=}
f
\overset{\mathbf X\ is\ odd\ or\ even}{\leq}
t
$,
whereas if $u=t$ then $\nega{u}=\nega{t}<t$ since the second row of Table~\ref{MasodiKK} shows that $\mathbf X$ is even.
Therefore, by denoting $y=\g{\nega{u}}{\nega{u}}$,
\begin{equation}\label{YiGYu}
y<\nega{u}
\end{equation}
holds since 
$y=\g{\nega{u}}{\nega{u}}
\overset{(\ref{JHGkcllvVjJKlc})}{\leq}
\g{\nega{u}}{t}=\nega{u}$ and equality cannot not hold since 
$\nega{u}$ is not idempotent.
Now $
\g{x}{\nega{u}}
\overset{(\ref{JHGkcllvVjJKlc})}{\leq}
\g{x}{t}=x$ follows.
Assume, by contradiction $\g{x}{\nega{u}}=x$.
It would imply $\g{x}{y}=\g{x}{(\g{\nega{u}}{\nega{u}})}=\g{(\g{x}{\nega{u}})}{\nega{u}}=x$, hence
by claim~\ref{eq_feltukrozes} in Lemma~\ref{tuttiINVOLUTIVE}, 
$\nega{\left(\g{\nega{x}}{\nega{y}}\right)}\geq\g{x}{y}=x$, and in turn
$\g{\nega{x}}{\nega{y}}\leq\nega{x}$ would follow on the one hand.
On the other hand, from
$\nega{y}
\overset{(\ref{YiGYu})}{>}
u\geq t$, by monotonicity $\g{\nega{x}}{\nega{y}}\geq\g{\nega{x}}{t}=\nega{x}$ follows, thus we obtain $\g{\nega{x}}{\nega{y}}=\nega{x}$, and hence $\nega{y}\leq \tau(\nega{x})
\overset{L\ref{tuttiINVOLUTIVE}\ref{NegArAAA}}{=}
u$, a contradiction to (\ref{YiGYu}).
We have just seen that $\g{x}{\nega{u}}<x$.
Next, assume that there exists $z\in X_u$ such that $\g{x}{\nega{u}}<z<x$ holds. 
Since $z<x$, $\g{x}{\nega{u}}\geq\nega{(\g{\nega{z}}{u})}=\nega{\nega{z}}=z$ follows by claim~\ref{eq_feltukrozes_CS} in Lemma~\ref{tuttiINVOLUTIVE}, a contradiction, so (\ref{SZORuVESSZO2}) is confirmed.

Next we show that $\mathbf X_u$ is even: 
$
\negaM{u}{u}
\overset{(\ref{MeGegyeznekPRIME})}{=}
\nega{u}=\g{u}{\nega{u}}
\overset{(\ref{SZORuVESSZO2})}{=}
u_{\downarrow_u}
<u
$.

Finally, we show that $\mathbf X_u$ is cancellative, by showing that 
every element of $X_u$ has inverse, that is,
for $x\in X_u$, $\gteu{x}{\negaM{u}{x}_{\uparrow_u}}=u$.
On the one hand, referring to (\ref{ValoBanAZ}),
$
\gteu{x}{\negaM{u}{x}_{\uparrow_u}}
>_u
\negaM{u}{u}
$
holds by residuation since $X$ is a chain.
It is equivalent to 
$
\gteu{x}{\negaM{u}{x}_{\uparrow_u}}
\geq_u
u
$
since $\mathbf X_u$ is even, yielding
$
\gteu{x}{\negaM{u}{x}_{\uparrow_u}}
\geq
u
$.
On the other hand, 
$
\gteu{x}{\negaM{u}{x}_{\uparrow_u}}
=
\g{x}{\negaM{u}{x}_{\uparrow_u}}
\overset{(\ref{MeGegyeznekPRIME})}{=}
\g{x}{\nega{x}_{\uparrow_u}}
\overset{(\ref{FelNeg_NegLe})}{=}
\g{x}{\nega{(x_{\downarrow_u})}}
\overset{(\ref{SZORuVESSZO2})}{=}
\g{x}{\nega{(\g{x}{\nega{u}})}}
\overset{L\ref{tuttiINVOLUTIVE}\ref{eq_quasi_inverse}}{=}
\g{x}{(\res{\te}{x}{u})}
\overset{\teALONE\ is\ residuated}{\leq}
u
$.

Summing up, $\mathbf X_u$ is a discretely ordered cancellative even involutive FL$_e$-chain.
\item
If $u\in\kappa_I$ then by Table~\ref{MasodiKK},
$u=t$ or $u\in\bar\kappa_I$.
In both cases $\nega{u}$ is idempotent.
We show that $\mathbf X_u$ is even:
first, $\negaM{u}{u}<u$ holds since 
if $u\in\bar\kappa_I$ then $u>t$ and hence $\negaM{u}{u}\overset{(\ref{MeGegyeznekPRIME})}{=}\nega{u}<\nega{t}\overset{L\ref{tuttiINVOLUTIVE}\ref{MiNdig}}{=}
f\overset{\mathbf X\ is\ odd\ or\ even}{\leq}
t<u$,
whereas if $u=t$ then $\mathbf X$ is even by the third row of Table~\ref{MasodiKK} and 
$\negaM{u}{u}\overset{(\ref{MeGegyeznekPRIME})}{=}
\nega{u}=\nega{t}\overset{L\ref{tuttiINVOLUTIVE}\ref{MiNdig}}{=}f
\overset{\mathbf X\ is\ even}{<}
t=u$, and second, by claims~\ref{item_boundary_zeta} and \ref{tau_lemma} in Lemma~\ref{tuttiINVOLUTIVE}, no element $x\in X$ such that $\nega{u}<x<u$ can be in $X_u$ (if $\nega{u}<x<u$ then $\nega{u}<\nega{x} <u$, so we may safely assume $x\geq t$, and then $x\in X_u$  implies $u=\tau(x)\leq x<u$, a contradiction); thus $\negaM{u}{u}=u_{\downarrow_u}$.
Summing up, $\mathbf X_u$ is an even involutive FL$_e$-chain with an idempotent falsum. 
It remains to prove $\gteu{x}{\negaM{u}{x}}=\negaM{u}{u}$ for $x\in X_u$, which follows from
$
\negaM{u}{u}
\overset{(\ref{MeGegyeznekPRIME})}{=}
\nega{u}=\nega{\tau(x)}=\nega{(\res{\te}{x}{x})}
\overset{L\ref{tuttiINVOLUTIVE}\ref{eq_quasi_inverse}}{=}
\g{x}{\nega{x}}
\overset{(\ref{MeGegyeznekPRIME})}{=}
\g{x}{\negaM{u}{x}}=\gteu{x}{\negaM{u}{x}}
$.
\end{description}
Next we prove that $\rho_{u\to v}$ is a homomorphism from the residuated lattice reduct of $\mathbf X_u$  to the residuated lattice reduct of $\mathbf X_v$.
Let $u,v\in X$ be positive idempotent elements of $\mathbf X$ such that $u<v$.
$\rho_{u\to v}$ maps $X_u$ to $X_v$ by claim~\ref{tau_lemma} in Lemma~\ref{tuttiINVOLUTIVE}.
$\rho_{u\to v}$ preserves the ordering since $\teALONE$ is monotone.
$\rho_{u\to v}$ preserves products since $\teALONE$ is associative and $v$ is idempotent. 
\\
To show that $\rho_{u\to v}$ preserves the residual operation we proceed as follows. 
Let $x,y\in X_u$.
It holds true that 
$
\g{\g{v}{\nega{(\g{x}{\nega{y}})}}}{\g{x}{\nega{(\g{v}{y})}}}
\overset{L\ref{tuttiINVOLUTIVE}\ref{eq_quasi_inverse}}{=}
\g{\g{x}{(\res{\te}{x}{y})}}{\g{(\res{\te}{y}{\nega{v})}}{v}}
\leq
f$ since $\teALONE$ is residuated,
hence by adjointness, 
$
\g{v}{\nega{(\g{x}{\nega{y}})}}
\leq
\nega{(\g{x}{\nega{(\g{v}{y})}})}
$
follows.
On the other hand, 
$
\g{v}{\nega{(\g{x}{\nega{y}})}}
\overset{L\ref{tuttiINVOLUTIVE}\ref{eq_feltukrozes}}{\geq}
\g{v}{(\g{\nega{x}}{y})}
=
\g{\nega{x}}{(\g{v}{y})}
$.
Now
$x=\g{t}{x}\leq\g{v}{x}$, and since $v>u=\tau(x)$ and hence $v$ does not stabilize $x$, $x<\g{v}{x}$ follows.
Therefore, 
$
\g{\nega{x}}{(\g{v}{y})}
\overset{L\ref{tuttiINVOLUTIVE}\ref{eq_feltukrozes_CS}}{\geq}
\nega{(\g{(\g{v}{x})}{\nega{(\g{v}{y})}})}
=
\nega{(\g{x}{(\g{v}{\nega{(\g{v}{y})}})})}
\overset{\nega{(\g{v}{y})}\in X_v}{=}
\nega{(\g{x}{\nega{(\g{v}{y})}})}
$.
Summing up, $\g{v}{\nega{(\g{x}{\nega{y}})}}=\nega{(\g{x}{\nega{(\g{v}{y})}})}$.
Therefore, 
$
\rho_{u\to v}(\res{u}{x}{y})
\overset{L\ref{tuttiINVOLUTIVE}\ref{eq_quasi_inverse}}{=}
\g{v}{\negaM{u}{(\gteu{x}{\negaM{u}{y}})}}
\overset{(\ref{MeGegyeznekPRIME})}{=}
\g{v}{\nega{(\g{x}{\nega{y}})}}
=
\nega{(\g{x}{\nega{(\g{v}{y})}})}
\overset{\nega{(\g{v}{y})}\in X_v}{=}
\nega{(\g{x}{(\g{v}{\nega{(\g{v}{y})}}}))}
=
\nega{(\g{(\g{v}{x})}{\nega{(\g{v}{y})}})}
\overset{(\ref{MeGegyeznekPRIME})}{=}
\negaM{v}{(\gtev{(\g{v}{x})}{\negaM{v}{(\g{v}{y})}})}
\overset{L\ref{tuttiINVOLUTIVE}\ref{eq_quasi_inverse}}{=}
\res{v}{(\g{v}{x})}{(\g{v}{y})}
=
\res{v}{\rho_{u\to v}(x)}{\rho_{u\to v}(y)}
$
an we are done.
\\
Finally, $\rho_{u\to v}$ preserves the unit element:
\begin{equation}\label{ksdnjkljKHJLKHJLJK}
\mbox{
$\rho_{u\to v}(u)=\g{v}{u}=v$ 
}
\end{equation}
holds since $v=\g{v}{v}\geq\g{v}{u}\geq\g{v}{t}=v$. 
Summing up, $\rho_{u\to v}$ is a homomorphism from the residuated lattice reduct of $\mathbf X_u$  to the residuated lattice reduct of $\mathbf X_v$.

\medskip\noindent
To conclude the proof of claim~(1) it only remains to prove
\begin{itemize}
\item[(A1):] 
For $u,v,w\in X$ positive idempotent elements such that $u\leq v\leq w$, and 
for $x\in X_u$, 
$
\rho_{u\to u}(x)=\g{u}{x}
\overset{L\ref{tuttiINVOLUTIVE}\ref{tnelnagyobb}}{=}
x
$
and
$
(\rho_{v\to w}\circ\rho_{u\to v})(x)
\overset{(\ref{EzARho})}{=}
\g{w}{(\g{v}{x})}
=
\g{(\g{w}{v})}{x}
\overset{(\ref{ksdnjkljKHJLKHJLJK})}{=}
\g{w}{x}
=
\rho_{u\to w}(x)
$.

\item[(A2):] 
Let $u,v\in X$ positive idempotent elements such that $u<v$.
By claim~\ref{eq_feltukrozes_CS} in Lemma~\ref{tuttiINVOLUTIVE},
$
v=
\g{v}{t}\geq
\g{v}{\nega{u}}\geq
\nega{(\g{\nega{v}}{v})}
=
\tau(v)
=v
$, 
yielding $\rho_{u\to v}(\nega{u})
\overset{(\ref{EzARho})}{=}
\g{v}{\nega{u}}=v\overset{(\ref{ksdnjkljKHJLKHJLJK})}{=}
\g{v}{u}
\overset{(\ref{EzARho})}{=}
\rho_{u\to v}(u)
$.
\end{itemize}

\bigskip
(2): We conclude by a series of claims.

$\leq$ is a total ordering on $X$. 
\begin{itemize}
\item 
Irreflexivity of $<$ is immediate form the irreflexivity of $<_\kappa$.
\item 
Connectedness of $<$ is obvious, too: 
either $\rho_{uv}(x)<_{uv}\rho_{uv}(y)$ or $\rho_{uv}(y)<_{uv}\rho_{uv}(x)$ or $\rho_{uv}(x)=\rho_{uv}(y)$ holds, since $\leq_{uv}$ is total.
The first two cases yield $x\overset{(\ref{RendeZes})}{<}y$ and $y\overset{(\ref{RendeZes})}{<}x$, respectively. If $\rho_{uv}(x)=\rho_{uv}(y)$ then either $u<_\kappa v$ or $v<_\kappa u$ or $u=_\kappa v$ holds since $\leq_\kappa$ is total.
Here $u=_\kappa v$ yields $x\overset{(A1)}{=}
\rho_{uu}(x)=\rho_{vv}(y)\overset{(A1)}{=}
y$,
whereas 
$u<_\kappa v$ and $v<_\kappa u$ yields $x\overset{(\ref{RendeZes})}{<}y$ and $y\overset{(\ref{RendeZes})}{<}x$, respectively.
\item 
$<$ is transitive:
Let $x\in X_u$, $y\in X_v$, $z\in X_w$, and assume $x<y<z$. 
From $x<y$ it follows that $\rho_{uv}(x)\leq_{uv}\rho_{uv}(y)$, hence by preservation of the ordering, $\rho_{uvw}(x)\leq_{uvw}\rho_{uvw}(y)$ holds. 
Analogously we obtain $\rho_{uvw}(y)\leq_{uvw}\rho_{uvw}(z)$, hence $\rho_{uvw}(x)\leq_{uvw}\rho_{uvw}(z)$ follows by the transitivity of $\leq_{uvw}$.
Therefore either $\rho_{uvw}(x)<_{uvw}\rho_{uvw}(z)$ and we conclude $x<z$, or
$\rho_{uvw}(x)=\rho_{uvw}(z)$.
The latter implies $\rho_{uvw}(x)=\rho_{uvw}(y)=\rho_{uvw}(z)$, and also $u<_\kappa v$ and $v<_\kappa w$.
Therefore, by the transitivity of $<_\kappa$, $u<_\kappa w$ follows and thus $x<z$.
\end{itemize}

$(X,\teALONE,t)$ is a commutative monoid. 
\begin{itemize}
\item 
Commutativity of $\teALONE$ is straightforward.
\item 
Let $x\in X_u$, $y\in X_v$, $z\in X_w$.
Then 
$
\g{(\g{x}{y})}{z}=
\g{(\gteM{uv}{\rho_{uv}(x)}{\rho_{uv}(y)})}{z} ,
$
and the latest is equal to
$
\gteM
{uvw}
{\rho_{uvw}\left(\gteM{uv}{\rho_{uv}(x)}{\rho_{uv}(y)}\right)}
{\rho_{uvw}(z)} 
$
since $\gteM{uv}{\rho_{uv}(x)}{\rho_{uv}(y)}\in X_{uv}$.
Since the $\rho$'s preserve products and $\cdot$ is idempotent,
the latest is equal to
$
\gteM
{uvw}
{\left(\gteM{uvw}
{\rho_{uvw}(x)}
{\rho_{uvw}(y)}\right)}
{\rho_{uvw}(z)}
.
$
Analogously follows that 
$\g{x}{(\g{y}{z})}$ is equal to 
$
\gteM
{uvw}
{\rho_{uvw}(x)}
{\left(
\gteM
{uvw}
{\rho_{uvw}(y)}
{\rho_{uvw}(z)}
\right)}
,
$
and hence the associativity of $\teALONE_{uvw}$ implies the associativity of $\teALONE$.
\item 
For $x\in X_u$,
$
\g{t}{x}
\overset{(\ref{EgySzeruTe})}{=}
\gteu{\rho_u(t)}{\rho_u(x)}
\overset{(\ref{P5})}{=}
\gteu{\rho_u(t)}{x}
\overset{(\ref{RHOhomo})}{=}
\gteu{u}{x}
=
x
$
holds using that $u$ is the unit element of $\mathbf X_u$.
\end{itemize}
$\komp$ is an order reversing bijection on $X$.
\begin{itemize}
\item[]
We start with two claims.
\item[-]
(C1)
If $u<_\kappa v$ then 
for $x\in X_u$, 
$\rho_v(\negaM{u}{x})
$ is the inverse of $\rho_v
(x)$ in $\mathbf X_v$.
Indeed, 
$
\gtev{\rho_v(x)}{\rho_v(\negaM{u}{x})}
\overset{(\ref{P5})}{=}
\gtev{\rho_{u\to v}(x)}{\rho_{u\to v}(\negaM{u}{x})}
\overset{(\ref{RHOhomo})}{=}
\rho_{u\to v}(\gteu{x}{\negaM{u}{x}})
$.
Note that $\gteu{x}{\negaM{u}{x}}=\negaM{u}{u}$ holds not only if $u\in\kappa_I$ (see (\ref{KiKiLesz})), but also if 
$u\in\kappa_o\cup\kappa_J$, since due to the cancellativity of $\mathbf X_u$,
$
\gteu{x}{\negaM{u}{x}}
=
\gteu{x}{(\res{u}{x}{\negaM{u}{u}})}
\overset{(\ref{EzAResiduum})}{=}
\gteu{x}{(\gteu{x^{-1_u}}{\negaM{u}{u}})}
=
\gteu{(\gteu{x}{x^{-1_u}})}{\negaM{u}{u}}
=
\gteu{u}{\negaM{u}{u}}
=
\negaM{u}{u}
$.
Therefore, 
$\rho_{u\to v}(\gteu{x}{\negaM{u}{x}})
=
\rho_{u\to v}(\negaM{u}{u})
$,
which is equal to $\rho_{u\to v}(u)
\overset{(\ref{RHOhomo})}{=}
v$
if $u\in\kappa_o$ (since then $\negaM{u}{u}=u$), and
is equal to
$
\overset{(A2)}{=}
\rho_{u\to v}(u)
\overset{(\ref{RHOhomo})}{=}
v
$
if $u\notin\kappa_o$.
\item[-]
(C2)
If $u<_\kappa v$ then for $x\in X_u$, $\negaM{v}{\rho_{v}(x)}=\rho_v(\negaM{u}{x})_{\downarrow_v}<\rho_v(\negaM{u}{x})$.
First we prove  
$
\gtev{\negaM{v}{\rho_{v}(x)}}{\rho_v(x)}
=
v_{\downarrow_v}
$:
if $v\in\kappa_J$ then
due to the cancellativity of $\mathbf X_v$ (see~(\ref{KiKiLesz})),
$
\gtev{\negaM{v}{\rho_{v}(x)}}{\rho_v(x)}
=
\gtev{\rho_v(x)}{(\res{v}{\rho_{v}(x)}{\negaM{v}{v}})}
\overset{(\ref{EzAResiduum})}{=}
\gtev{\rho_v(x)}{(\gtev{\rho_v(x)^{-1}}{\negaM{v}{v}})}
=
\gtev{(\gtev{\rho_v(x)}{\rho_v(x)^{-1}})}{\negaM{v}{v}}
=\gtev{v}{\negaM{v}{v}}=\negaM{v}{v}
\overset{\mathbf X_v \ is \ even,\ see~(\ref{KiKiLesz})}{=}
v_{\downarrow_v}
$,
whereas if $v\in\kappa_I$ then 
$\gtev{\negaM{v}{\rho_{v}(x)}}{\rho_v(x)}=\negaM{v}{v}=v_{\downarrow_v}$
holds by (\ref{KiKiLesz}), and we are done. 
Now, multiplying both sides by $\rho_v(\negaM{u}{x})$ yields
$\negaM{v}{\rho_{v}(x)}=\gtev{v_{\downarrow_v}}{\rho_v(\negaM{u}{x})}$ using (C1).
If $v\in\kappa_J$ then by (\ref{KiKiLesz}) we can apply claim~\ref{eGGyelLejjebB} in Lemma~\ref{tuttiINVOLUTIVE} resulting in
$
\rho_v(\negaM{u}{x})_{\downarrow_v}
=
\gtev{v_{\downarrow_v}}{\rho_v(\negaM{u}{x})}
<
\gtev{v}{\rho_v(\negaM{u}{x})}
=
\rho_v(\negaM{u}{x})
$, so we are done.
If $v\in\kappa_I$ then by (\ref{KiKiLesz}), $\mathbf X_v$ is an even involutive FL$_e$-chain with an idempotent falsum. 
Therefore, $\mathbf X_v=Sp(\mathbf A,\mathbf H)$ by Theorem~\ref{sPliT}, and
since $\rho_v(\negaM{u}{x})$ is invertible by (C1), 
$\rho_v(\negaM{u}{x})$ is an element of $\mathbf H$ by claim~(\ref{UNICITYofH}) in Theorem~\ref{sPliT}.
On the other hand, $v_{\downarrow_v}$ is clearly in $ H^\bullet$.
Therefore, 
$
\gtev{v_{\downarrow_v}}{\rho_v(\negaM{u}{x})}
\overset{(\ref{ProdSplit})}{=}
(\gtev{v}{\rho_v(\negaM{u}{x})})_{\downarrow_v}
=
\rho_v(\negaM{u}{x})_{\downarrow_v}
\overset{\rho_v(\negaM{u}{x})\in H}{<}
\rho_v(\negaM{u}{x})
$
holds.
\end{itemize}
\begin{itemize}
\item 
Since for $u\in \kappa$,  $\kompM{u}$ is of order $2$, so is $\komp$ over $X$ by (\ref{IgyKellEztCsinalni}), hence $\komp$ is a bijection.
It remains to prove that $\komp$ is order reversing.
Let
$X_u\ni x\leq y\in X_v$.
If $u=v$ then $x\leq_u y$ holds by (\ref{RendeZes}), hence $\negaM{u}{y}\leq_u \negaM{u}{x}$ follows since $\mathbf X_u$ is involutive, thus $\nega{y}\leq\nega{x}$ holds by (\ref{RendeZes}) and (A1). 
If $u<_\kappa v$ then 
$x\leq y$ implies
$\rho_v(x)\leq_{uv} y$ by (\ref{RendeZes}), hence $\negaM{v}{y}\leq_v\negaM{v}{\rho_{v}(x)}$ follows since $\mathbf X_v$ is involutive.
Therefore, 
$
\negaM{v}{y}
\leq_v
\negaM{v}{\rho_{v}(x)}
\overset{(C2)}{<_v}
\rho_v(\negaM{u}{x})
$,
and thus
$\negaM{v}{y}<_v\rho_v(\negaM{u}{x})$  implies 
$\negaM{v}{y}\leq\negaM{u}{x}$
by (\ref{RendeZes}),
yielding
$\nega{y}\leq\nega{x}$
by (\ref{IgyKellEztCsinalni}). 
If $u>_\kappa v$ then $x\leq y$ implies $x<_u\rho_u(y)$ by (\ref{RendeZes}),
which is equivalent to 
$\negaM{u}{\rho_u(y)}<_u\negaM{u}{x}$
since $\mathbf X_u$ is involutive.
By (C2), $\rho_u(\negaM{v}{y})\leq_u\negaM{u}{x}$ follows, which yields
$\negaM{v}{y}\leq\negaM{u}{x}$
by (\ref{RendeZes}), and hence $\nega{y}\leq\nega{x}$ follows by (\ref{IgyKellEztCsinalni}).
\end{itemize}

$\nega{t}$ is a dualizing element of $(X,\leq,\te)$.
\begin{itemize}
\item
It suffices to prove that for $x\in X$, there exists $\res{\te}{x}{\nega{t}}$ and 
\begin{equation}\label{kjkjKKGHKHKHljklkLJJKLd}
\res{\te}{x}{\nega{t}}=\nega{x}
,
\end{equation}
since the involutivity of $\komp$ then ensures $\res{\te}{(\res{\te}{x}{\nega{t}})}{\nega{t}}=x$.
Equivalently, that for $x,y\in X$, 
$
\g{x}{y}\leq \nega{t} \mbox{ \ if and only if \ } x\leq\nega{y}
$.
Let $x\in X_u$, $y\in X_v$.
Since $\teALONE$ is commutative and since $\komp$ is an order reversing bijection, we may safely assume 
$u\leq_\kappa v$.
Since
$
\nega{t}
\overset{(\ref{IgyKellEztCsinalni})}{=}
\negaM{t}{t}
\in X_t$
and $t\leq_\kappa v$,
by (\ref{EgySzeruTe}) $\g{x}{y}\leq \negaM{t}{t}$ is equivalent to
$
\gteM{v}{\rho_{v}(x)}{\rho_{v}(y)}
\leq
\negaM{t}{t}
.
$
Since $\gteM{v}{\rho_{v}(x)}{\rho_{v}(y)}\in X_{v}$,
$\gteM{v}{\rho_{v}(x)}{\rho_{v}(y)}\leq \negaM{t}{t}$ is equivalent to 
$\gteM{v}{\rho_{v}(x)}{\rho_{v}(y)}\leq_{v}\negaM{v}{v}$:
indeed, if $t=v$ then $\gteM{v}{\rho_{v}(x)}{\rho_{v}(y)}\leq \negaM{t}{t}$ is equivalent to 
$\gteM{v}{\rho_{v}(x)}{\rho_{v}(y)}
\overset{(\ref{RendeZes})}{\leq_{v}}
\negaM{t}{t}
=
\negaM{v}{v}
$,
whereas 
if $t<_\kappa v$ then $v\notin\kappa_o$ and $\gteM{v}{\rho_{v}(x)}{\rho_{v}(y)}\leq \negaM{t}{t}$ is equivalent to 
$\gteM{v}{\rho_{v}(x)}{\rho_{v}(y)}
\overset{(\ref{RendeZes})}{<_{v}}
\rho_v(\negaM{t}{t})
\overset{(\ref{P5})}{=}
\rho_{t\to v}(\negaM{t}{t})
\overset{(A2)}{=}
\rho_{t\to v}(t)
\overset{(\ref{RHOhomo})}{=}
v
$,
that is, equivalent to
$\gteM{v}{\rho_{v}(x)}{\rho_{v}(y)}
\overset{v\notin\kappa_o,\,(\ref{KiKiLesz})}{\leq_{v}}
v_{\downarrow_v}
\overset{(\ref{KiKiLesz})}{=}
\negaM{v}{v}$.
Since $\mathbf X_v$ is residuated, $\gteM{v}{\rho_{v}(x)}{\rho_{v}(y)}\leq_{v}\negaM{v}{v}$ is equivalent to 
$\rho_{v}(x)\leq_v \res{v}{y}{\negaM{v}{v}}=\negaM{v}{y}$, and by (\ref{RendeZes}) it is equivalent to $x\leq\negaM{v}{y}\overset{(\ref{IgyKellEztCsinalni})}{=}\nega{y}$.
\end{itemize}

Summing up, we have shown that $(X,\leq)$ is a chain and $(X,\teALONE,t)$ is a commutative monoid.
Since $\nega{t}$ is a dualizing element of $(X,\leq,\te)$, 
Lemma~\ref{SimpLER} shows that $(X,\leq,\te)$ is residuated
and 
$\res{\te}{x}{y}=\nega{\left(\g{x}{\nega{y}}\right)}$. 
Since $\komp$ (given in (\ref{IgyKellEztCsinalni}))
coincides with the residual complement of $\mathbf X$ (given by $\res{\te}{x}{\nega{t}}$, see (\ref{kjkjKKGHKHKHljklkLJJKLd})), 
and since $\komp$ is an order reserving involution on $X$, it follows that $\mathbf X$ is involutive.
Finally, 
by (\ref{KiKiLesz}),
$\mathbf X_t$ and hence also $\mathbf X$ is odd if $t\in\kappa_o$, $\mathbf X$ is 
even with a non-idempotent falsum if $t\in\kappa_J$, and $\mathbf X$ is even with an idempotent falsum if $t\in\kappa_I$.

\bigskip
(3):
Let $\mathcal A=\langle \mathbf X_u, \rho_{u\to v} \rangle_{\boldsymbol\kappa}$ be a bunch of layer algebras, and adapt the 
notations in Definition~\ref{DEFbunch2} and the definitions in claim~(2).
To see that the universe of the $u^{\rm th}$ layer algebra of $\mathcal A_{\left(\mathcal X_{\mathcal A}\right)}$ is equal to $X_u$ which is the universe  of the $u^{\rm th}$ layer algebra $\mathbf X_u$ of $\mathcal A$,
we need to prove that for $x\in X$ (where $X$ is given in (\ref{EZazX})), 
$\res{\te}{x}{x}=u$ if and only if $x$ is in $X_u$, the universe of the $u^{th}$-layer algebra:
$x\in X$ implies that $x\in X_v$ for some $v\in\kappa$. 
Now
$
u
=
\res{\te}{x}{x}
\overset{(\ref{IgYaReSi})}{=}
\nega{(\g{x}{\nega{x}})}
\overset{(\ref{IgyKellEztCsinalni})\,(\ref{EgySzeruTe})}{=}
\nega{(\gtev{x}{\negaM{v}{x}})}
\overset{(\ref{IgyKellEztCsinalni})}{=}
\negaM{v}{(\gtev{x}{\negaM{v}{x}})}$,
where the last equality holds since $X_v$ is closed under $\kompM{v}$ and $\tev$ and thus $\gtev{x}{\negaM{v}{x}}$ is in $X_v$.
Therefore, $u\in X_v$ follows.
Hence $u=v$ must hold since $u\in X_u$ and $X$ is the disjoint union of the $X_u$'s by (\ref{EZazX}).
 The definition of the ordering relation in claim~(1) and (\ref{RendeZes}) show that the ordering of the $u^{\rm th}$ layer algebra of $\mathcal A_{\left(\mathcal X_{\mathcal A}\right)}$ is the same as the ordering $\leq_u$ of the $u^{\rm th}$ layer algebra $\mathbf X_u$ of $\mathcal A$.
Likewise show (\ref{EgySzeruTe}) and the definition of the monoidal operation in claim~(1) that the monoidal operation of the $u^{\rm th}$ layer algebra of $\mathcal A_{\left(\mathcal X_{\mathcal A}\right)}$ is the same as the monoidal operation $\teu$ of the $u^{\rm th}$ layer algebra $\mathbf X_u$ of $\mathcal A$.
Since both the $u^{\rm th}$ layer algebra of $\mathcal A_{\left(\mathcal X_{\mathcal A}\right)}$ and the $u^{\rm th}$ layer algebra $\mathbf X_u$ of $\mathcal A$ are involutive FL$_e$-chains over the same universe, equipped with the same ordering relation and the same product operation, their residual operations -- which are uniquely determined by these -- must coincide, too. 
The unit element of $\mathbf X_u$ is $u$, therefore $u$ acts as the unit element of the $u^{\rm th}$ layer algebra of $\mathcal A_{\left(\mathcal X_{\mathcal A}\right)}$ (which is over the same set $X_u$ and is equipped with the same monoidal operation, as we have seen above).
Finally, the falsum constant of $\mathbf X_u$ is $\negaM{u}{u}$. On the other hand, the falsum constant of the $u^{\rm th}$ layer algebra of $\mathcal A_{\left(\mathcal X_{\mathcal A}\right)}$ is 
$
\res{\te}{u}{\nega{u}}
\overset{(\ref{IgyKellEztCsinalni})}{=}
\res{\te}{u}{\negaM{u}{u}}
=
\res{u}{u}{\negaM{u}{u}}
=
\negaM{u}{u}
$,
where the last equality holds by residuation since $u$ is the unit element over $X_u$.
Summing up, $\mathcal A_{\left(\mathcal X_{\mathcal A}\right)}=\mathcal A$.

\smallskip
Let $\mathbf X=(X,\leq,\teALONE,\ite{\te},t,f)$ be an odd or even involutive FL$_e$-chain, and adapt the definitions in claim~(1).
We have seen in the proof of claim~(1) that the $X_u$'s are nonempty. It is straightforward that they are disjoint, too, and their union, which is the universe of  $\mathcal X_{\mathcal A_\mathbf X}$, see (\ref{EZazX}), is equal to $X$.
To prove that the ordering of $\mathbf X$ and of $\mathcal X_{\mathcal A_\mathbf X}$ coincide, first we prove the following statement.
If $v< u$ and $x\in X_v$ then
\begin{equation}\label{HgHJJhGkhj}
\rho_u(x)=\min\{z\in X_u: z\geq x\}
>x
.
\end{equation}
Indeed, 
$\g{u}{x}\in X_u$ by claim~\ref{tau_lemma} in Lemma~\ref{tuttiINVOLUTIVE}, and
$\g{u}{x}\geq x$ holds since $\g{u}{x}\geq\g{t}{x}=x$.
By contradiction, assume that there exists $z\in X_u$
such that 
$x<z<\g{u}{x}$.
Since $\leq$ is total, by adjointness $\nega{\nega{z}}<\g{u}{x}$ is equivalent to
$f<\g{\nega{z}}{(\g{u}{x})}=\g{(\g{\nega{z}}{u})}{x}
\overset{\nega{z}\in X_u}{=}
\g{\nega{z}}{x}$.
Finally, since $\leq$ is total, by adjointness $f<\g{\nega{z}}{x}$ is equivalent to
$x>\nega{\nega{z}}=z$, a contradiction.
Referring to (\ref{HgHJJhGkhj}) a moment's reflection shows that the ordering $\leq$ of $X$ coincides with the ordering of $\mathcal X_{\mathcal A_\mathbf X}$ given by (\ref{RendeZes}). 
Since for $x\in X_u$ and $y\in X_v$,
$
\g{x}{y}
\overset{L\ref{tuttiINVOLUTIVE}\ref{tnelnagyobb}}{=}
\g{(\g{\g{u}{u}}{x})}{(\g{\g{v}{v}}{y})}=
\g{(\g{(\g{u}{v})}{x})}{(\g{(\g{u}{v})}{y})}
=
\g{\rho_{uv}(x)}{\rho_{uv}(y)}
=\gteM{uv}{\rho_{uv}(x)}{\rho_{uv}(y)}
$,
the monoidal operation $\teALONE$ (of $X$) coincides with $\teALONE$ given in (\ref{EgySzeruTe}).
Since both $\mathbf X$ and $\mathcal X_{\mathcal A_\mathbf X}$ are involutive FL$_e$-chains over the same universe, equipped with the same ordering relation and the same product operation, their residual operations -- which are uniquely determined by these -- must coincide, too. 
By (\ref{EQezLeszAgUUUUU}), the unit element of $\mathbf X_t$ is $t$, and hence the unit element of 
$\mathcal X_{\mathcal A_\mathbf X}$ is also $t$.
Finally, the falsum constant of $\mathbf X$ is $f$, hence the falsum constant of $\mathbf X_t$ is $\negaM{t}{t}=\res{\te}{t}{\nega{t}}\overset{L\ref{tuttiINVOLUTIVE}\ref{MiNdig}}{=}
\res{\te}{t}{f}\overset{L\ref{tuttiINVOLUTIVE}\ref{MiNdig}}{=}f$.
Therefore, the falsum constant of $\mathcal X_{\mathcal A_\mathbf X}$ is also $\nega{t}=f$.
Summing up, $\mathcal X_{\mathcal A_\mathbf X}=\mathbf X$.
\end{proof}

\begin{remark}
Two elements of the construction of $\mathcal X_{\mathcal A}$ in Lemma~\ref{BUNCHalg_X}/(\ref{dfgdkjkjHKJhKJNHG}) are reminiscent to the P\l{}onka sum construction \cite{plonka}.
In the P\l{}onka sum construction the universe of the algebra is the disjoint union of the universes of the algebras in the direct system. So is in our construction, c.f.\,(\ref{EZazX}).
In the P\l{}onka sum construction every binary operator $\circ\in\{\teALONE, \ite{\te},\wedge,\vee\}$ would be given in terms of the direct system as follows: for $x\in X_u$, $y\in X_v$, 
\begin{equation}\label{plonkaWAY}
x\circ y
=
\rho_{u\to uv}(x)\circ_{uv}\rho_{u\to uv}(y)
\end{equation}
In our construction it is the case only for $\circ\in\{\teALONE\}$, c.f.\,(\ref{EgySzeruTe}). 
As for $\circ\in\{\ite{\te}\}$,
if $\ite{\te}$ were given by (\ref{plonkaWAY}) then 
for any positive idempotent elements $u<v$, it would hold true that
$
\res{\te}{v}{u}
=
\res{\te}{u}{v}
$
since
$
\res{\te}{v}{u}
\overset{(\ref{plonkaWAY})}{=}
\res{\te_{v}}{v}{v}
\overset{(\ref{plonkaWAY})}{=}
\res{\te}{u}{v}
$.
But in an involutive FL$_e$-chain 
it
would be equivalent to
$\g{v}{\nega{u}}=\g{u}{\nega{v}}$
by Lemma~\ref{tuttiINVOLUTIVE}/\ref{eq_quasi_inverse},
a contradiction, since
$\g{v}{\nega{u}}=v$
and
$\g{u}{\nega{v}}=\nega{v}$
(as shown by
$
v
\overset{L\ref{tuttiINVOLUTIVE}\ref{ZetaOfIdempotent}}{=}
\tau(v)
\overset{L\ref{tuttiINVOLUTIVE}\ref{eq_quasi_inverse}}{=}
\nega{(\g{v}{\nega{v}})}
\overset{L\ref{tuttiINVOLUTIVE}\ref{eq_feltukrozes_CS}}{\leq}
\g{v}{\nega{u}}\leq\g{v}{t}=v
$
and
$
\nega{v}
=
\g{t}{\nega{v}}
\leq
\g{u}{\nega{v}}
\leq
\g{v}{\nega{v}}
\overset{L\ref{tuttiINVOLUTIVE}\ref{eq_quasi_inverse}}{=}
\nega{\tau(v)}
\overset{L\ref{tuttiINVOLUTIVE}\ref{ZetaOfIdempotent}}{=}
\nega{v}
$, respectively).
Finally, notice that P\l{}onka's definition doesn't work for $\circ\in\{\wedge,\vee\}$ since (\ref{plonkaWAY}) would render different elements of the universe equal.
Indeed, if $u<v$ are positive idempotent elements, $x\in X_u$ and $\rho_{u\to v}(x)=y\in X_v$ then according to (\ref{plonkaWAY}), 
$x\wedge y=\rho_{u\to v}(x)\wedge_v y=y=\rho_{u\to v}(x)\vee_v y=x\vee y$ yielding $x=y$,
a contradiction since $x$ and $y$ are in different layers. 
Even though our definition is applied in our specific (residuated lattice) setting only,
we have introduced in (\ref{RendeZes}) {\em a general way of ordering P\l{}onka sums}, call it the {\em directed lexicographic order}.
Making use of the (slight) similarity between our construction and that of P\l{}onka's would have made the related proof only a few lines shorter. Therefore, we have included those few lines, too, to make the treatment self-contained. 
\end{remark}

Bunches of layer algebras will, in turn, be represented by bunches of layer groups in Section~\ref{DEFbunch}. To that end, in the following two sections the necessary auxiliary results (along with some which are more general than what we actually need in this paper) will be developed.

\section{Even involutive FL$_e$-chains with non-idempotent falsum constants -- Changing the falsum constant}\label{WaRMuPEVEN_NID}

An abelian $o$-group is called {\em discrete}, if there exists the smallest positive element greater than the unit element.
It is equivalent to saying that the abelian $o$-group is discretely ordered, or that its induced 
cancellative odd involutive FL$_e$-algebra is discretely ordered.
In the following definition the residuated chain reduct of the FL$_e$-chain is left unchanged, and only the falsum constant, and thus also the residual complement are changed slightly.
\begin{definition}\rm
For 
a discretely ordered, cancellative, odd, involutive FL$_e$-chain $\mathbf X=(X, \leq,\teALONE, \ite{\te}, t, t)$, 
let
$$
\mathbf X_{\contour{black}{$_\downarrow$}}
=
(X, \leq,  \ite{\te}, t, t_\downarrow).
$$
For a 
discretely ordered,
cancellative, even, involutive FL$_e$-chain
$\mathbf Y=(Y, \leq,\teALONE, \ite{\te}, t, t_\downarrow)$, 
let
$$\mathbf Y_{\contour{black}{$_\uparrow$}}=
(Y, \leq,\teALONE, \ite{\te}, t, t).$$
\end{definition}

\begin{lemma}\label{PROdownshift}
Let $\mathbf X$ and $\mathbf Y$ be discretely ordered, cancellative, involutive FL$_e$-chains, $\mathbf X$ is odd, $\mathbf Y$ is even.
Then
\begin{enumerate}
\item 
$\mathbf X_{\contour{black}{$_\downarrow$}}$ is a discretely ordered, cancellative, even, involutive FL$_e$-chain,
\item 
$\mathbf Y_{\contour{black}{$_\uparrow$}}$ is a discretely ordered, cancellative, odd, involutive FL$_e$-chain,
\item
$\mathbf X_{\contour{black}{$_{\downarrow\uparrow}$}}=\mathbf X$
and
$\mathbf Y_{\contour{black}{$_{\uparrow\downarrow}$}}=\mathbf Y$.
\end{enumerate}
\end{lemma}
\begin{proof}
The rest being obvious we prove the involutivity of $\mathbf X_{\contour{black}{$_\downarrow$}}$ and $\mathbf Y_{\contour{black}{$_\uparrow$}}$.
(1): To see that $\mathbf X_{\contour{black}{$_\downarrow$}}$ is involutive, 
denote the residual complement operation of $\mathbf X$ by $\komp$, and the inverse operation of $\lambda(\mathbf X)$ by ${}^{-1}$. 
Then
$
\res{}{x}{t_\downarrow}
\overset{(\ref{EzAResiduum})}{=}
\g{x^{-1}}{t_\downarrow}
\overset{(\ref{EzAzInverz})}{=}
\g{\nega{x}}{t_\downarrow}
\overset{L\ref{tuttiINVOLUTIVE}\ref{eGGyelLejjebB}}{=}
(\nega{x})_\downarrow
$,
and hence (\ref{FelNeg_NegLe}) confirms involutivity.
(2):
To see that $\mathbf Y_{\contour{black}{$_\uparrow$}}$ is involutive, 
denote the residual complement operation of $\mathbf Y$ by $\kompM{\bullet}$. 
Then
$
\res{}{y}{t}
\overset{L\ref{tuttiINVOLUTIVE}\ref{eq_quasi_inverse}}{=}
\negaM{\bullet}{(\g{y}{\negaM{\bullet}{t}})}
\overset{L\ref{tuttiINVOLUTIVE}\ref{MiNdig}}{=}
\negaM{\bullet}{(\g{y}{f})}
\overset{\mathbf Y \ is \ even}{=}
\negaM{\bullet}{(\g{y}{t_\downarrow})}
\overset{L\ref{tuttiINVOLUTIVE}\ref{eGGyelLejjebB}}{=}
\negaM{\bullet}{{y_\downarrow}}$ holds,
and hence (\ref{FelNeg_NegLe}) confirms involutivity.
\end{proof}

\section{Even involutive FL$_e$-algebras with idempotent falsum constants -- subgroup splits of odd involutive FL$_e$-algebras}\label{WaRMuPEVEN_ID}

As an investigation into the structure of residuated lattices which are not necessarily totally ordered, a one-to-one correspondence between pairs of an odd involutive FL$_e$-algebra and a cancellative subalgebra of it, and even involutive FL$_e$-algebras where the residual complement of the unit element is idempotent will be proved in this section.

In groups the unit element has two different roles to play. It serves as the unit element of the multiplication, and also the product of any element by its inverse is equal to it.
We shall replace the unit element of any odd involutive FL$_e$-algebra by two elements, 
each will inherit a single role. This way both the unit element itself and its two roles will be \lq\lq split\rq\rq\  into two. 
Some anticipatory examples: the unit-split algebra of the one-element group will be the two-element Boolean algebra, and 
the unit-split algebra of the Sugihara lattice $\mathbf S_{\mathbb Z}$ will be {\em the} Sugihara lattice $\mathbf S_{\mathbb Z^\star}$, 
(both named by Meyer, see \cite{AB}[p.~414]).
Moreover, not only the unit element, but in fact any subgroup of an odd involutive FL$_e$-algebra can be \lq\lq split\rq\rq, that is, each element of the subgroup of an odd involutive FL$_e$-algebra will be replaced by two elements.
We prove that by splitting (and thus \lq\lq doubling\rq\rq) a subgroup of any odd involutive FL$_e$-algebra we obtain an even involutive FL$_e$-algebra with an idempotent falsum constant, and each even involutive FL$_e$-algebra with an idempotent falsum constant arises this way in a unique manner (Theorem~\ref{sPliT}).

Such a splitting method has roots in the literature. In \cite{casari} an analogous idea has been applied to abelian pregroups. Our method here is both more general and more specific: the algebras that we split are more general than that of \cite{casari}, however, in our approach the (second) algebra we use in the splitting procedure is the two element Boolean algebra (each element is split into two elements only) whereas in \cite{casari} it is taken from the more general class of zero-pregroups.
The notion of the melting structure (of a pregroup) in \cite{casari} is analogous to the image (of the even involutive FL$_e$-algebra with idempotent falsum constant) under what we call its canonical homomorphism.
The splitting procedure in \cite{casari} is precursor also for the partial lex-product construction of \cite{JS_Hahn}, which when further generalized to partial sublex-products has been capable to describe the structure of all odd, involutive FL$_e$-chains which have only finitely many idempotent elements \cite{JS_Hahn,Jenei_Hahn_err}.

\begin{definition}\label{SubgroupSplitREV}
Let $\mathbf Y=(Y,\wedge, \vee,\star,\ite{\star},t,f)$ be an even FL$_e$-algebra with an idempotent falsum constant. Denote its residual complement by $\kompM{\star}$.
Let
$$
\mbox{
$\pi_1(\mathbf Y)=\mathbf X=(X,\dot \wedge, \dot \vee,\teALONE,\ite{\te},t,t)$
and 
$\pi_2(\mathbf Y)=\mathbf H=(H,\dot \wedge, \dot \vee,\teALONE,\ite{\te},t,t)$
}
$$
be given by
\begin{equation}\label{EQmiaH}
\mbox{
$H=\{x\in Y : \gstar{x}{f}<x\}$,
$ H^\bullet=\{\gstar{x}{f} : x\in H\}$, 
$X=Y\setminus H^\bullet$,
}
\end{equation}
for $y\in Y$,
\begin{equation}\label{canHOMmegint}
h_\mathbf Y(x)=
\left\{
\begin{array}{ll}
x			& \mbox{ if $x\in X$}\\
x_\uparrow	& \mbox{ if $x\in  H^\bullet$}\\
\end{array}
\right. 
,
\end{equation}
and for $x,y\in X$, 
$$
x\dot \wedge y=h_\mathbf Y(x\wedge y),
$$
$$
x\dot \vee y=h_\mathbf Y(x\vee y),
$$
\begin{equation}\label{HHKgffLjljLhjgkH}
\g{x}{y}=h_\mathbf Y(\gstar{x}{y}),
\end{equation}
\begin{equation}\label{HHKJKHJKJJcgfslLJH}
\res{\te}{x}{y}=h_\mathbf Y(\res{\star}{x}{y}).
\end{equation}
\end{definition}

\begin{definition}
Let $\mathbf X=(X,\wedge, \vee,\teALONE,\ite{\te},t,t)$ be an odd involutive FL$_e$-algebra with residual complement $\komp$.
Let
$\mathbf H\leq\mathbf X$ (over $H\subseteq X$), $\mathbf H$ cancellative\footnote{Equivalently, $\lambda(\mathbf H)$ is a lattice ordered abelian group by Lemma~\ref{inducedDEF}.}.
Let $Sp(\mathbf X, \mathbf H)$, 
the $\mathbf H$-split of $\mathbf X$ 
be
$$
\mathbf Y=(Y,\wedge_Y,\vee_Y,\star,\ite{\star},t, t^\bullet)
,
$$ 
where
$ H^\bullet=\{ h^\bullet : h\in H\}$ is a copy of $H$ disjoint from $X$,
\begin{equation}\label{EzAzy}
Y=X\cup  H^\bullet
,
\end{equation}
the lattice ordering $\leq$ of $X$ is extended to $Y$ by
letting 
\begin{equation}\label{KibovitettRendezes}
\mbox{
$ a^\bullet<_Y b^\bullet$ and $x<_Y a^\bullet<_Yy$
for $a,b\in H$,  $a<b$, $x,y\in X$, $x<a\leq y$
,
}
\end{equation}
$h : Y \to X$, 
\begin{equation}\label{canHOM}
h(x)=\left\{
\begin{array}{ll}
x			& \mbox{ if $x\in X$}\\
x_\uparrow	& \mbox{ if $x\in  H^\bullet$}\\
\end{array}
\right. ,
\end{equation}
where $_\uparrow$ and $_\downarrow$ denote the neighbour operations of $\mathbf Y$,
\begin{equation}\label{ProdSplit}
\gstar{x}{y}=\left\{
\begin{array}{ll}
\g{h(x)}{h(y)}		& \mbox{ if $\g{h(x)}{h(y)}\notin H$ or $x,y\in H$}\\
(\g{h(x)}{h(y)})_\downarrow	& \mbox{ if $\neg(x,y\in H)$ and $\g{h(x)}{h(y)}\in H$}\\
\end{array}
\right. ,
\end{equation}
$\kompM{\star} : Y\to Y$,
\begin{equation}\label{SplitNega}
\negaM{\star}{x}=\left\{
\begin{array}{ll}
\nega{x}	& \mbox{ if $x\in X\setminus H$}\\
(\nega{x})_\downarrow	& \mbox{ if $x\in H$}\\
\nega{(x_\uparrow)}		& \mbox{ if $x\in  H^\bullet$}\\
\end{array}
\right. ,
\end{equation}
$$
\res{\star}{x}{y}=\negaM{\star}{(\gstar{x}{\negaM{\star}{y}})}
.
$$
\end{definition}

\begin{theorem}\label{sPliT}
The following statements hold true.
\begin{enumerate}
\item
Let $\mathbf X$ be an odd involutive FL$_e$-algebra, $\mathbf H$ be a cancellative subalgebra of it, and $\mathbf Y$ be $Sp(\mathbf X, \mathbf H)$, the $\mathbf H$-split of $\mathbf X$. Then
\begin{enumerate}
\item\label{HkjasHHLH} 
$\mathbf Y$ is an even FL$_e$-algebra with an idempotent falsum constant.
\item\label{AzAtloAzFalsum}
If $\mathbf X$ is cancellative then 
for $x\in Y$, $\gstar{x}{\negaM{\star}{x}}= t^\bullet$ holds.
\item\label{UNICITYofH}
$\{x\in Sp(\mathbf X, \mathbf H) : x \mbox{ is invertible in $Sp(\mathbf X, \mathbf H)$} \}=
\mathbf H$.
\end{enumerate}

\item
Let $\mathbf Y$ be an even involutive FL$_e$-algebra with an idempotent falsum constant.
Then 
\begin{enumerate}
\item\label{VisszIsMEGyeget} 
there exists a unique pair $\langle\mathbf X, \mathbf H\rangle$ of an odd involutive FL$_e$-algebra
$\mathbf X=(X,\dot \wedge, \dot \vee,\teALONE,\ite{\te},t,t)$
and a cancellative subalgebra $\mathbf H$ of $\mathbf X$ 
such that
$\mathbf Y$ is $Sp(\mathbf X, \mathbf H)$, the $\mathbf H$-split of  $\mathbf X$.
$\mathbf X$ and $\mathbf H$ are given by $\pi_1(\mathbf Y)$ and $\pi_2(\mathbf Y)$, respectively.

\item\label{HaAkkorKancellativ}
If for $x\in Y$, $\gstar{x}{\negaM{\star}{x}}=f$ holds then $\mathbf X$ is cancellative.
\item\label{hSpid}
$h_\mathbf Y$ is a surjective homomorphism (called the {\em canonical} homomorphism of $\mathbf Y$) from $\mathbf Y$ onto $\mathbf X$, that is, it holds true that
$$\mathbf X=h_\mathbf Y(Sp(\mathbf X, \mathbf H)).$$
\end{enumerate}
\end{enumerate}
\end{theorem}
\begin{proof}
(\ref{HkjasHHLH}):
It is obvious from (\ref{KibovitettRendezes}) that
\begin{equation}\label{PontAlatta}
\mbox{
for $a\in H$, $a$ is the unique cover of $ a^\bullet$
}
.
\end{equation}
Therefore, by letting
$k : X \to Y$, 
$$
k(x)=\left\{
\begin{array}{ll}
x			& \mbox{ if $x\in X\setminus H$}\\
 x^\bullet	& \mbox{ if $x\in H$}\\
\end{array}
\right. ,
$$
a moment's reflection shows that $\leq_Y$ is a lattice ordering on $Y$ and the corresponding lattice operations are given by 
$$
\mbox{
$
x\wedge_Yy=y\wedge_Yx
=\left\{
\begin{array}{ll}
y					& \mbox{ if $h(y)\leq h(x)$, $y\in H^\bullet$}\\
h(x)\wedge h(y)			& \mbox{ otherwise}\\
\end{array}
\right. 
,
$
}
$$
$$
\mbox{
$
x\vee_Yy=y\vee_Yx
=\left\{
\begin{array}{ll}
h(x)\vee h(y)		& \mbox{ if $h(x)\vee h(y)\notin H$}\\
k(h(x)\vee h(y))		& \mbox{ if $\{h(x),h(y)\}\not\ni h(x)\vee h(y)\in H$}\\
h(x)\vee h(y)		& \mbox{ if $h(x)\leq h(y)$, $y\in H$}\\
k(h(x)\vee h(y))		& \mbox{ if $h(x)\leq h(y)$, $y\in  H^\bullet$}\\
\end{array}
\right. 
.
$
}
$$
\\
It is straightforward from (\ref{ProdSplit}) that $\star$ is commutative.
$t$ is the unit element of $\star$: using that $t$ is the unit element over $X$,
$$
\gstar{x}{t}
\overset{(\ref{ProdSplit}),\,(\ref{canHOM})}{=}
\left\{
\begin{array}{ll}
\g{h(x)}{t}=h(x)=x		& \mbox{ if $h(x)\notin H$ or $x\in H$}\\
(\g{h(x)}{t})_\downarrow=h(x)_\downarrow=x	& \mbox{ if $x\notin H$ and $h(x)\in H$}\\
\end{array}
\right. .
$$
We obtain $ a^\bullet=a_{\downarrow}<a$ for $a\in H$ from (\ref{PontAlatta}), 
therefore, $\kompM{\star}$ in (\ref{SplitNega}) is clearly an order reversing involution by (\ref{FelNeg_NegLe}).

\smallskip
As for the associativity of $\star$, first notice that 
\begin{equation}\label{Hacsakigynem}
\mbox{
$\gstar{x}{y}\in H$ if and only if $x,y\in H$.
}
\end{equation}
Indeed, if $x,y\in H$ then 
$\gstar{x}{y}\overset{(\ref{ProdSplit})}{=}
\g{h(x)}{h(y)}
\overset{(\ref{canHOM})}{=}
\g{x}{y}\in H
$ since $H$ is closed under $\teALONE$. 
If $\neg(x,y\in H)$ then either
$\g{h(x)}{h(y)}\in H$ in which case
$\gstar{x}{y}\overset{(\ref{ProdSplit})}{\in} H^\bullet$ implies $\gstar{x}{y}\notin H$,
or 
$\g{h(x)}{h(y)}\notin H$ in which case
$\gstar{x}{y}\overset{(\ref{ProdSplit})}{=}
\g{h(x)}{h(y)}\notin H$, hence we are done.
Next notice that
\begin{equation}\label{wefekjhkjkjajkhklHhkl}
\mbox{
for $x,y\in Y$, $h(\gstar{x}{y})=\g{h(x)}{h(y)}$.
}
\end{equation}
Indeed, 
if the first row of (\ref{ProdSplit}) defines the value of $\gstar{x}{y}$ then 
$\gstar{x}{y}=(\g{h(x)}{h(y)})_\downarrow$
and $\g{h(x)}{h(y)}\in H$, and hence
$
h(\gstar{x}{y})=
h((\g{h(x)}{h(y)})_\downarrow)
\overset{(\ref{canHOM})}{=}
(\g{h(x)}{h(y)})_{\downarrow\uparrow}
=
\g{h(x)}{h(y)}$,
whereas 
if 
$\gstar{x}{y}
\overset{(\ref{ProdSplit})}{=}
\g{h(x)}{h(y)}$
then
$
h(\gstar{x}{y})=
h(\g{h(x)}{h(y)})=
\g{h(x)}{h(y)}
$
since $h$ maps to $X$, $X$ is closed under $\teALONE$, and $h$ is the identity on $X$.\\
Now,  
$h(\gstar{(\gstar{x}{y})}{z})=h(\gstar{x}{(\gstar{y}{z})})$ readily follows from (\ref{wefekjhkjkjajkhklHhkl}) and the associativity of $\teALONE$.
Therefore, by (\ref{canHOM}), $\gstar{(\gstar{x}{y})}{z}\not=\gstar{x}{(\gstar{y}{z})}$ can only be possible if one side is in $H$ and the other side is in $ H^\bullet$.
However, (\ref{Hacsakigynem}) shows that if one side is in $H$ then $x,y,z\in H$, and hence, since $H$ is closed under $\star$ by (\ref{Hacsakigynem}), also the other side must be in $H$.

\smallskip
It is easily seen that 
\begin{equation}\label{POwrtleqY}
\mbox{
$\star$ is isotone with respect to  $\leq_Y$.
}
\end{equation}
Indeed, let $x,y,z\in Y$ with $y<_Y z$.
Since $h$ and $\teALONE$ are increasing with respect to $\leq$, it follows that $\g{h(x)}{h(y)}\leq\g{h(x)}{h(z)}$.
Now by (\ref{ProdSplit}), $\gstar{x}{y}\leq_Y\gstar{x}{z}$ clearly holds if $\g{h(x)}{h(y)}<\g{h(x)}{h(z)}$, hence we may assume $\g{h(x)}{h(y)}=\g{h(x)}{h(z)}$.
But then the only way for $\gstar{x}{y}\not\leq_Y\gstar{x}{z}$ to hold is if 
$
\gstar{x}{y}=\g{h(x)}{h(y)}
$
and 
$
\gstar{x}{z}=
(\g{h(x)}{h(z)})_\downarrow
$,
which by (\ref{ProdSplit}) leads to assuming
$\neg(x,z\in H)$, $\g{h(x)}{h(z)}\in H$
and
either 
$\g{h(x)}{h(y)}\notin H$ or $x,y\in H$.
It follows that $x,y\in H$.
But then $\g{h(x)}{h(y)}=\g{x}{y}<\g{x}{h(z)}\leq\g{h(x)}{h(z)}$ is a contradiction, where the strict inequality follows from $y<h(z)$ and since $x$ has an inverse.

\smallskip
To prove that $\star$ is residuated, and that
$
\res{\star}{x}{y}=\negaM{\star}{(\gstar{x}{\negaM{\star}{y}})},
$
by 
Lemma~\ref{SimpLER}
it suffices to verify that $ t^\bullet$ is a dualizing element of $(Y,\leq,\star)$. 
Since $\kompM{\star}$ is clearly an order reversing involution, it suffices to verify that that $\res{\star}{x}{t_\downarrow}$ exists and is equal to $\negaM{\star}{x}$.
It amounts to verifying only three cases.
(i)
If $x\in H$ then $\negaM{\star}{x}=(\nega{x})_\downarrow\in  H^\bullet$ and 
$\gstar{x}{\negaM{\star}{x}}
\overset{(\ref{SplitNega})}{=}
\gstar{x}{(\nega{x})_\downarrow}
\overset{(\ref{ProdSplit})}{=}
(\g{x}{\nega{x})}_\downarrow=t_\downarrow$, where in the last equality we used that $\mathbf H$ is a subgroup, hence $\nega{x}$ is the inverse of $x$ in $\mathbf X$.
On the other hand, for $z>\negaM{\star}{x}$ it follows that $z\geq\nega{x}$ and hence 
$\gstar{x}{z}
\overset{(\ref{POwrtleqY})}{\geq}
\gstar{x}{\nega{x}}
\overset{\nega{x}\in H,\,(\ref{ProdSplit})}{=}
\g{x}{\nega{x}}=t>t_\downarrow$.
Therefore, by residuation, $\res{\star}{x}{t_\downarrow}$ exists and is equal to $\negaM{\star}{x}$.
(ii)
If $x\in H^\bullet$ then $x=y_\downarrow$ for some $y\in H$ and $\negaM{\star}{x}=\nega{y}=\nega{(x_\uparrow)}
\overset{(\ref{FelNeg_NegLe})}{=}
\nega{x}_\downarrow>\nega{x}$.
We obtain 
$\gstar{x}{\negaM{\star}{x}}=\gstar{y_\downarrow}{\nega{y}}
\overset{(\ref{ProdSplit})}{=}
(\g{y}{\nega{y})}_\downarrow=t_\downarrow$.
On the other hand, for $z>\negaM{\star}{x}$ it follows that 
$z\geq\nega{x}$
and hence 
$\gstar{x}{z}
\overset{(\ref{POwrtleqY})}{\geq}
\gstar{y_\downarrow}{\nega{x}}\overset{(\ref{ProdSplit})}{=}(\g{y}{\nega{x}})_\downarrow=(\g{y}{\nega{y}_\uparrow})_\downarrow$.
Here $\g{y}{\nega{y}_\uparrow}>t$ holds by residuation since $\nega{y}_\uparrow>\nega{y}$, hence 
$\gstar{x}{z}>t_\downarrow$ follows.
(iii)
Finally, if $x\in X\setminus H$ then $\negaM{\star}{x}=\nega{x}\in X\setminus H$.
The case $\gstar{x}{\negaM{\star}{x}}=\gstar{x}{\nega{x}}\in H$ leads to
$\gstar{x}{\negaM{\star}{x}}=\gstar{x}{\nega{x}}\overset{(\ref{ProdSplit})}{=}(\g{x}{\nega{x})}_\downarrow\leq t_\downarrow$, whereas 
if $\gstar{x}{\negaM{\star}{x}}=\gstar{x}{\nega{x}}\in X\setminus H$ then 
$\gstar{x}{\negaM{\star}{x}}=\gstar{x}{\nega{x}}\overset{(\ref{ProdSplit})}{=}\g{x}{\nega{x}}\leq t$,
but due to $\gstar{x}{\negaM{\star}{x}}\in X\setminus H$ and $t\in H$ equality cannot hold, hence here too, $\gstar{x}{\negaM{\star}{x}}\leq t_\downarrow$ follows.
On the other hand if $z>\negaM{\star}{x}=\nega{x}$ then $\gstar{x}{z}\geq(\g{x}{z})_\downarrow>t_\downarrow$ holds by (\ref{ProdSplit}) and by residuation, respectively. 
\\
Summing up, the falsum-free reduct of $\mathbf Y$ is an involutive commutative residuated lattice with residual complement operation $\kompM{\star}$. 
By 
the second row of (\ref{SplitNega}), $\negaM{\star}{t}=(\nega{t})_\downarrow=t_\downarrow$, and $t_\downarrow$ is idempotent by the second row of (\ref{ProdSplit}). 
A particular instance of (\ref{PontAlatta}) shows that $t$ is the unique cover of $ t^\bullet$, hence $ t^\bullet=t_\downarrow$ and thus $\mathbf Y$ is even. 

\medskip
(\ref{AzAtloAzFalsum}):
Since 
\begin{equation}\label{IgyJoBb}
\g{h(x)}{h(\nega{x})}=\g{x}{\nega{x}}\overset{L\ref{tuttiINVOLUTIVE}\ref{canc_Fremegy}}{=}
f=t\in H,
\end{equation}
for $x\in Y$,
$
\gstar{x}{\negaM{\star}{x}}
\overset{(\ref{SplitNega})}{=}
$
$$
\small
\left\{
\begin{array}{ll}
\gstar{x}{\nega{x}}
\overset{(\ref{ProdSplit}),\,(\ref{IgyJoBb})}{=}
(\g{h(x)}{h(\nega{x})})_\downarrow
= t^\bullet
& \mbox{ if $x\in X\setminus H$}\\
\gstar{x}{(\nega{x})_\downarrow}
\overset{(\ref{ProdSplit})}{=}
(\g{h(x)}{h(\nega{x}_\downarrow)})_\downarrow
\overset{\nega{x}\in H,\,(\ref{canHOM})}{=}
(\g{x}{\nega{x}})_\downarrow
\overset{(\ref{EzAResComplement})}{=}
(\g{x}{x^{-1}})_\downarrow
= t^\bullet
& \mbox{ if $x\in H$}\\
\gstar{x}{\nega{(x_\uparrow)}}
\overset{(\ref{ProdSplit})}{=}
(\g{h(x)}{h(\nega{x_\uparrow})})_\downarrow
\overset{(\ref{canHOM})}{=}
(\g{x_\uparrow}{\nega{x_\uparrow}})_\downarrow
\overset{(\ref{EzAResComplement})}{=}
(\g{x_\uparrow}{{x_\uparrow}^{-1}})_\downarrow
= t^\bullet
& \mbox{ if $x\in  H^\bullet$}\\
\end{array}
\right. .
$$

\medskip
(\ref{UNICITYofH}):
Since $t\in H$ and $ H^\bullet$ is disjoint from $H$, referring to (\ref{canHOM}), it follows from (\ref{ProdSplit}) that $\gstar{x}{y}$ can be equal to $t$ only if $x,y\in H$. On the other hand, every element of $H$ is invertible in $\mathbf X$ by claim~\ref{csoportLesz} in Lemma~\ref{tuttiINVOLUTIVE}, since $\mathbf H$ is cancellative by assumption.
Hence
$
\gstar{x}{x^{-1}}
\overset{(\ref{ProdSplit})}{=}
\g{h(x)}{h(x^{-1})}
\overset{(\ref{canHOM})}{=}
\g{x}{x^{-1}}
=t
$.

\bigskip
(\ref{VisszIsMEGyeget}):
Unicity of $\mathbf H$ follows from claim~(\ref{UNICITYofH}), and it readily implies the unicity of $\mathbf X$, too, by (\ref{EzAzy}).
Let $\mathbf X=\pi_1(\mathbf Y)$ and $\mathbf H=\pi_2(\mathbf Y)$.
Denote $ x^\bullet=\gstar{x}{f}$ for $x\in H$.

\begin{enumerate}[(i)]
\item\label{Hinvertible}
Any element $x$ of $H$ is invertible, that is, $\gstar{x}{\negaM{t}{x}}=t$ holds, where $\kompM{t}$ is given by $$\negaM{t}{x}=\res{\star}{x}{t}.$$
Indeed, $t\geq\gstar{x}{(\res{\star}{x}{t})}
\overset{L\ref{tuttiINVOLUTIVE}\ref{eq_quasi_inverse}}{=}
\gstar{x}{\negaM{\star}{(\gstar{x}{\negaM{\star}{t}})}}
\overset{L\ref{tuttiINVOLUTIVE}\ref{MiNdig}}{=}
\gstar{x}{\negaM{\star}{(\gstar{x}{f})}}\not\leq f$,
where the latest step holds by residuation since 
$x\in H$, that is 
$\gstar{x}{f}<x$, and it implies $\negaM{\star}{(\gstar{x}{f})}>\negaM{\star}{x}$. 
Since $\mathbf Y$ is even, $t\geq\gstar{x}{\negaM{t}{x}}\not\leq f$ implies 
$\gstar{x}{\negaM{t}{x}}=t$.
\item\label{NincsMetszet}
$H\cap H^\bullet=\emptyset$. Indeed, if $x\in H^\bullet$, that is, if $x=\gstar{y}{f}$ for some $y\in H$ then 
$\gstar{x}{f}=\gstar{(\gstar{y}{f})}{f}=\gstar{y}{(\gstar{f}{f})}=\gstar{y}{f}=x$ ensures $x\notin H$.
\item\label{subsetHOOD}
$t\in H\subseteq X$. Indeed, $t\in H$ because of $\gstar{t}{f}=f<t$, and $H\subseteq X$ readily follows from (\ref{EQmiaH}) and claim~\ref{NincsMetszet}.
Hence it holds true that
\begin{equation}\label{cupPONT}
Y=(X\setminus H)
\overset{\cdot}{\cup}
H
\overset{\cdot}{\cup}
 H^\bullet.
\end{equation}

\item\label{AlaLovokk}
Next we prove $ m^\bullet=m_\downarrow<m$ for $m\in H$. 
Indeed, the assumption $\gstar{m}{f}<z<m$ would yield 
$f=
\gstar{t}{f}
\overset{\claim \ref{Hinvertible}}{=}
\gstar{(\gstar{\negaM{t}{m}}{m})}{f}=
\gstar{\negaM{t}{m}}{(\gstar{m}{f})}
\overset{L\ref{tuttiINVOLUTIVE}\ref{diagonalCOSinverse}}{<}
\gstar{\negaM{t}{m}}{z}
\overset{L\ref{tuttiINVOLUTIVE}\ref{diagonalCOSinverse}}{<}
\gstar{\negaM{t}{m}}{m}=t$, a contradiction to $t$ covering $f$.
\item\label{nyilazaS}
For $m\in H$ and $y\in (X\setminus H)\cup H^\bullet$ it holds true that
$$\gstar{ m^\bullet}{y}=\gstar{m}{y}.$$
Indeed, $(X\setminus H)\cup H^\bullet=Y\setminus H$ holds by (\ref{cupPONT}), therefore
$\gstar{y}{f}=y$. 
We obtain 
$
\gstar{m}{y}=\gstar{m}{(\gstar{y}{f})}=\gstar{(\gstar{m}{f})}{y}=\gstar{ m^\bullet}{y}$, as stated.
\item\label{OTTazINVERZhBAN}
For $m\in H$, $\negaM{t}{m}\in H$:
$\negaM{t}{m}\notin H$ would imply $
t
\overset{\claim \ref{Hinvertible}}{=}
\gstar{m}{\negaM{t}{m}}
\overset{\claim \ref{nyilazaS}}{=}
\gstar{ m^\bullet}{\negaM{t}{m}}
=
\gstar{(\gstar{m}{f})}{\negaM{t}{m}}
=
\gstar{(\gstar{m}{\negaM{t}{m}})}{f}
\overset{\claim \ref{Hinvertible}}{=}
\gstar{t}{f}
=
f$,
a contradiction.
\end{enumerate}

\noindent
The following product table holds true\footnote{We shall refer to the $(i,j)$ cell of this product table by $\star_{(i,j)}$}, see Table~\ref{Reszeik}.
\begin{table}[h]
\begin{center}
\caption{}
\begin{tabular}{ccccc}
$\star$ & \vline  & $y\in X\setminus H$ & $l\in H$  & $ l^\bullet\in H^\bullet$ \\
\hline
$x\in X\setminus H$  & \vline & $\in(X\setminus H)\cup H^\bullet$ & $\in X\setminus H$ & $\gstar{x}{l}\in X\setminus H$ \\
$m\in H$  & \vline & \color{midgrey}$\in(X\setminus H)$  & $\in H$ & $(\gstar{m}{l})^\bullet\in H^\bullet$ \\
$ m^\bullet\in H^\bullet$ & \vline & \color{midgrey}$\gstar{m}{y}\in X\setminus H$ & \color{midgrey}$(\gstar{m}{l})^\bullet\in H^\bullet$ & $(\gstar{m}{l})^\bullet\in H^\bullet$  \\
\label{Reszeik}
\end{tabular}
\end{center}
\end{table}
\item[$\star_{(2,2)}$:]
Since $m$ is invertible by claim~\ref{Hinvertible}, $\gstar{(\gstar{m}{l})}{f}=\gstar{m}{(\gstar{l}{f})}<\gstar{m}{l}$ follows by claim~\ref{diagonalCOSinverse} in Lemma~\ref{tuttiINVOLUTIVE}.
\item[$\star_{(2,3)}$:]
$\gstar{m}{ l^\bullet}=\gstar{m}{(\gstar{l}{f})}=\gstar{(\gstar{m}{l})}{f}
\overset{\star_{(2,2)}}{=}
(\gstar{m}{l})^\bullet\in H^\bullet$.
\item[$\star_{(3,3)}$:]
$\gstar{ m^\bullet}{ l^\bullet}=\gstar{(\gstar{m}{f})}{(\gstar{l}{f})}=\gstar{(\gstar{m}{l})}{(\gstar{f}{f})}=\gstar{(\gstar{m}{l})}{f}
\overset{\star_{(2,2)}}{=}
(\gstar{m}{l})^\bullet
\in H^\bullet$.
\item[$\star_{(1,2)}$:]
By (\ref{cupPONT}) the opposite of the statement is $
\gstar{x}{l}\in H\cup H^\bullet$.
Then, by claims~\ref{Hinvertible} and \ref{OTTazINVERZhBAN},
$x=
\gstar{x}{t}=
\gstar{x}{(\gstar{l}{\negaM{t}{l}})}=
\gstar{(\gstar{x}{l})}{\negaM{t}{l}}
\in
\gstar{( H\cup H^\bullet)}{H}\subseteq H\cup H^\bullet$
follows using $\star_{(2,2)}$ and $\star_{(2,3)}$, a contradiction to (\ref{cupPONT}).
\item[$\star_{(1,3)}$:] It follows from claim~\ref{nyilazaS} and $\star_{(1,2)}$.
\item[$\star_{(1,1)}$:]
Since $y\notin H$, $y=\gstar{y}{f}$ holds. 
Therefore,
$\gstar{x}{y}=\gstar{x}{(\gstar{y}{f})}=\gstar{(\gstar{x}{y})}{f}$ follows, hence $\gstar{x}{y}$ cannot be in $H$.

\medskip\noindent
We are ready to prove that $\mathbf X=(X,\dot \wedge, \dot \vee,\teALONE,\ite{\te},t,t)$ is an odd involutive FL$_e$-algebra.
($X,\dot \wedge,\dot \vee)$ is a lattice.
Indeed, all elements of $ H^\bullet$ are meet-irreducible because of claim~\ref{AlaLovokk}, hence $X$ is closed under $\dot \wedge$ (clearly, $\dot \wedge$ is the restriction of $\wedge$ to $X$).
Commutativity of $\dot \vee$ is straightforward, and using claim~\ref{AlaLovokk} a moment's reflection shows that $\dot \vee$ is associative, too, and the absorption law holds for $\dot \vee$ and $\dot \wedge$. 
Commutativity of $\teALONE$ is straightforward. 
$X$ is closed under $\teALONE$ since by claim~\ref{AlaLovokk}, $(\gstar{x}{y})_\uparrow\in H$ if $\gstar{x}{y}\in H^\bullet$.
$t\in X$ holds by claim~\ref{subsetHOOD}.
Since $t\in H$, $\star_{(2,1)}$ and $\star_{(2,2)}$
show that for $y\in X$,  $\gstar{t}{y}\notin H^\bullet$ holds, hence
$
\g{t}{y}
\overset{(\ref{HHKgffLjljLhjgkH})}{=}
h_\mathbf Y(\gstar{t}{y})
\overset{(\ref{canHOMmegint})}{=}
\gstar{t}{y}
$
and thus
$t$ is the unit element for $\teALONE$ over $X$ since it is the unit element for $\star$ over $Y$.
\\
As for the associativity of $\teALONE$, notice that 
\begin{equation}\label{MegintHhh}
\mbox{
for $x,y\in Y$, $h_\mathbf Y(\gstar{x}{y})=\g{h_\mathbf Y(x)}{h_\mathbf Y(y)}$ holds.
}
\end{equation}
Indeed, by (\ref{HHKgffLjljLhjgkH}), $\g{h_\mathbf Y(x)}{h_\mathbf Y(y)}=h_\mathbf Y(\gstar{h_\mathbf Y(x)}{h_\mathbf Y(y)})$, and 
Table~\ref{Reszeik} readily confirms that $h_\mathbf Y(\gstar{h_\mathbf Y(x)}{h_\mathbf Y(y)})=h_\mathbf Y(\gstar{x}{y})$.
Hence, $(X,\teALONE)$ being the homomorphic image of a semigroup, is a semigroup.
\\
Next we prove that $\teALONE$ is residuated.
For $x,y\in X$, 
$
\res{\te}{x}{y}=
\max\{z\in X:\g{x}{z}\leq y\}
$.
Here
$$
\g{x}{z}
\overset{(\ref{HHKgffLjljLhjgkH})}{=}
\left\{
\begin{array}{ll}
\gstar{x}{z}				& \mbox{ if $\gstar{x}{z}\notin H^\bullet$}\\
(\gstar{x}{z})_\uparrow		& \mbox{ if $\gstar{x}{z}\in H^\bullet$}\\
\end{array}
\right.
.
$$
If $\gstar{x}{z}\in H^\bullet$ then since $y\in X$, $(\gstar{x}{z})_\uparrow\leq y$ holds if and only if $\gstar{x}{z}\leq y$ holds by claim~\ref{AlaLovokk}.
Therefore, 
$\max\{z\in X:\g{x}{z}\leq y\}=\max\{z\in X:\gstar{x}{z}\leq y\}$ holds yielding
$\res{\te}{x}{y}=\res{\star}{x}{y}$.
Since $\res{\te}{x}{y}\in X$, it also follows that $\res{\te}{x}{y}\overset{(\ref{canHOMmegint})}{=}
h_\mathbf Y(\res{\te}{x}{y})=h_\mathbf Y(\res{\star}{x}{y})$, as stated.
\\
Involutivity of $\kompM{t}$ is seen as follows.
We will verify that
$$
\mbox{
$\negaM{t}{x}=\negaM{\star}{x}_\uparrow$ if $x\in H$, and
$\negaM{t}{x}=\negaM{\star}{x}$ if $x\in X\setminus H$.
}
$$
This, combined with (\ref{FelNeg_NegLe}) and that $\negaM{t}{x}\in H$ if and only if $x\in H$ (shown by claims~\ref{OTTazINVERZhBAN} and \ref{Hinvertible}) concludes the proof of the statement. 
Clearly, $\negaM{t}{x}=\res{\star}{x}{t}\geq\res{\star}{x}{f}=\negaM{\star}{x}$.
Let $x\in H$. 
Equality cannot hold since $\gstar{x}{\negaM{t}{x}}=t$ by claim~\ref{Hinvertible}, whereas $\gstar{x}{\negaM{\star}{x}}\leq f<t$ by residuation and since $\mathbf Y$ is even.
Assume that there exits $a\in X$ such that 
$\negaM{t}{x}>a>\negaM{\star}{x}$.
Then
$
\negaM{t}{x}
=
\res{\star}{x}{t}
\overset{L\ref{tuttiINVOLUTIVE}\ref{eq_quasi_inverse}}{=}
\negaM{\star}{(\gstar{x}{\negaM{\star}{t}})}
\overset{L\ref{tuttiINVOLUTIVE}\ref{MiNdig}}{=}
\negaM{\star}{(\gstar{x}{f})}
>
a
>
\negaM{\star}{x}
$,
and hence
$\gstar{x}{f}<\negaM{\star}{a}<x$, a contradiction to claim~\ref{AlaLovokk}.
Let $x\in X\setminus H$. 
If $\negaM{t}{x}>\negaM{\star}{x}$ then, as above, $\gstar{x}{f}<x$ follows, a contradiction to $x\notin H$.
\\
Finally, $\mathbf X$ is clearly odd, since the constant which defines the involution $\kompM{t}$ is the unit element.

\medskip\noindent
Next we prove that $\mathbf H$ is a cancellative subalgebra of $\mathbf X$.
Indeed, $H$ is closed under $\star$, shown by $\star_{(2,2)}$, 
$t\in H$ holds by claim \ref{subsetHOOD},
and $\mathbf H$ has an inverse operation $\kompM{t}$, see claims~\ref{Hinvertible} and \ref{OTTazINVERZhBAN},
hence $\mathbf H$ is cancellative by claim~\ref{csoportLesz} in Lemma~\ref{tuttiINVOLUTIVE}.
It is a subalgebra of $\mathbf X$ by claim~\ref{subsetHOOD}.

\medskip\noindent
\item
Finally, we verify that
$\mathbf Y$ is the $\mathbf H$-split of $\mathbf X$.
Indeed, the universe is as expected, see (\ref{cupPONT}).
The elements of $ H^\bullet$ are just below the respective elements of $H$, as they should be, see claim~\ref{AlaLovokk}.
Finally we verify that $\star$ coincides with the product operation of the $\mathbf H$-split of $\mathbf X$.
\begin{description}
\item
If $x,y\in H$ then $\gstar{x}{y}\in H$ by $\star_{(2,2)}$. Hence $\gstar{x}{y}\notin  H^\bullet$ by claim~\ref{NincsMetszet}, yielding $\gstar{x}{y}\overset{(\ref{HHKgffLjljLhjgkH})}{=}\g{x}{y}\overset{(\ref{canHOMmegint})}{=}\g{h_\mathbf Y(x)}{h_\mathbf Y(y)}$, as required in (\ref{ProdSplit}).
\item
If $\g{h_\mathbf Y(x)}{h_\mathbf Y(y)}\notin H$ then since $h_\mathbf Y(x),h_\mathbf Y(y)\in X$ and $X$ is closed under $\teALONE$, $\g{h_\mathbf Y(x)}{h_\mathbf Y(y)}\in X\setminus H$ follows. Therefore, by $\star_{(1-1,2-2)}$\footnote{This notation refers to the submatrix containing $\star_{(1,1)}$, $\star_{(1,2)}$, $\star_{(2,1)}$, $\star_{(2,2)}$.}, either $h_\mathbf Y(x)$ or $h_\mathbf Y(y)$ (say $h_\mathbf Y(x)$) must be in $X\setminus H$, yielding $h_\mathbf Y(x)\overset{(\ref{canHOMmegint})}{=}x\in X\setminus H$. 
By the first row of Table~\ref{Reszeik}, for any $y\in Y$, $\gstar{x}{y}=\gstar{x}{h_\mathbf Y(y)}$.
Since $H\not\ni\g{h_\mathbf Y(x)}{h_\mathbf Y(y)}
\overset{(\ref{MegintHhh})}{=}
h_\mathbf Y(\gstar{x}{y})$, it follows that $\gstar{x}{y}\notin H^\bullet$ and hence
$\gstar{x}{h_\mathbf Y(y)}\notin H^\bullet$.
Therefore, $\gstar{x}{y}=\gstar{x}{h_\mathbf Y(y)}\overset{(\ref{canHOMmegint})}{=}h_\mathbf Y(\gstar{x}{h_\mathbf Y(y)})\overset{(\ref{HHKgffLjljLhjgkH})}{=}\g{x}{h_\mathbf Y(y)}\overset{(\ref{canHOMmegint})}{=}\g{h_\mathbf Y(x)}{h_\mathbf Y(y)}$, as required in (\ref{ProdSplit}). 
\item
Assume $\neg(x,y\in H)$ and $\g{h_\mathbf Y(x)}{h_\mathbf Y(y)}\in H$.
By (\ref{MegintHhh}), $h_\mathbf Y(\gstar{x}{y})\in H$, hence $\gstar{x}{y}\in H\cup H^\bullet$. Since $\neg(x,y\in H)$, by Table~\ref{Reszeik} it follows that $\gstar{x}{y}\in  H^\bullet$.
Table~\ref{Reszeik} and $\gstar{x}{y}\in  H^\bullet$ also implies that either $x,y\in X\setminus H$, or at least one of $x$ and $y$ is in $ H^\bullet$ and the other is in $H\cup H^\bullet$. In all these cases, by Table~\ref{Reszeik} and (\ref{canHOMmegint}) it follows that
$\gstar{x}{y}=\gstar{h_\mathbf Y(x)}{h_\mathbf Y(y)}$.
Hence $\gstar{h_\mathbf Y(x)}{h_\mathbf Y(y)}\in H^\bullet$ and it yields 
$
(\gstar{h_\mathbf Y(x)}{h_\mathbf Y(y)})_\uparrow
\overset{(\ref{canHOMmegint})}{=}
h_\mathbf Y(\gstar{h_\mathbf Y(x)}{h_\mathbf Y(y)})
\overset{(\ref{HHKgffLjljLhjgkH})}{=}
\g{h_\mathbf Y(x)}{h_\mathbf Y(y)}
$,
that is,
$\gstar{x}{y}=\gstar{h_\mathbf Y(x)}{h_\mathbf Y(y)}=(\g{h_\mathbf Y(x)}{h_\mathbf Y(y)})_\downarrow$, as required in (\ref{ProdSplit}).
\end{description}

\smallskip
(\ref{HaAkkorKancellativ}):
Since $h_\mathbf Y$ maps onto $\mathbf X$, 
to prove that $\mathbf X$ is cancellative it suffices to prove that each element of $h_\mathbf Y(\mathbf Y)$ has inverse, which holds since for $x\in Y$, 
$
\g{h_\mathbf Y(x)}{h_\mathbf Y(\negaM{\star}{x})}
\overset{(\ref{MegintHhh})}{=}
h_\mathbf Y(\gstar{x}{\negaM{\star}{x}})
=
h_\mathbf Y(f)
\overset{\mathbf Y \ is \  even}{=}
h_\mathbf Y(t_\downarrow)
\overset{t\in H,\,(\ref{canHOM})}{=}
t
$.

\medskip
(\ref{hSpid}):
It is obvious that $h_\mathbf Y$ preserves the unit element and the falsum constant. 
The definitions in (\ref{canHOMmegint})-(\ref{HHKJKHJKJJcgfslLJH}) readily yield that $h_\mathbf Y$ preserves the meet, the join, the product and the residual operation.
\end{proof}

\section{Bunches of layer algebras vs.\,bunches of layer groups}\label{COnsTRUctioN}

We introduce the notion of bunches of layer groups, and show that every bunch of layer algebras can be represented by a unique bunch of layer groups.

\begin{definition}\label{DEFbunch}
Let $(\kappa,\leq_\kappa)$ be a totally ordered set with least element $t$, and
let an ordered triple
$\langle \bar\kappa_I, \bar\kappa_J, \{t\}\rangle$
be a partition of $\kappa$, where $\bar\kappa_I$ and $\bar\kappa_J$ can also be empty. 
Define $\kappa_o$, $\kappa_J$, and $\kappa_I$ by one of the rows of Table~\ref{ThetaPsiOmega}, 
and let 
$\boldsymbol\kappa=\langle \kappa_o, \kappa_J, \kappa_I,\leq_\kappa\rangle$. 
Let
$\textbf{\textit{G$_u$}}=(G_u,\preceq_u,\cdot_u,\ { }^{-1_u},u)$
be a family of abelian $o$-groups indexed by elements of $\kappa$, 
an let 
$\textbf{\textit{H$_u$}}=(H_u,\preceq_u,\cdot_u,\ { }^{-1_u},u)$
be a family of abelian $o$-groups indexed by elements of $\kappa_I$, 
such that 
\begin{equation}\label{DiSCRetE}
\mbox{
for $u\in\kappa_J$, $\textbf{\textit{G$_u$}}$ is discrete,
}
\end{equation}
$$
\mbox{
for $u\in\kappa_I$, $\textbf{\textit{H$_u$}}\leq\textbf{\textit{G$_u$}}$,
}
$$
and 
such that
for $u,v\in\kappa$, $u\leq_\kappa v$, there exist
$$
\mbox{
homomorphisms $\varsigma_{u\to v} : G_u\to G_v$ 
}
$$
satisfying
\begin{itemize}
\item[(G1)] 
$\varsigma_{u\to u}=id_{G_u}$ and $\varsigma_{v\to w}\circ\varsigma_{u\to v}=\varsigma_{u\to w}$ \hfill (direct system property),
\item[(G2)]
for $v>_\kappa u\in\kappa_J$,
$\varsigma_{u\to v}(u)=\varsigma_{u\to v}(u_{\downarrow_u})$,
\item[(G3)]
for $u<_\kappa v\in\kappa_I$,
$\varsigma_{u\to v}$ maps into $H_v$.

\end{itemize}
Call 
${\mathcal G}=\langle \textbf{\textit{G$_u$}},\textbf{\textit{H$_u$}}, \varsigma_{u\to v} \rangle_{\boldsymbol\kappa}$
a {\em bunch of layer groups}.
\end{definition}

\begin{lemma}\label{BUNCHalg_BUNCHgroup}
The following statements hold true.
\begin{enumerate}
\item 
Given a bunch of layer algebras 
${\mathcal A}=\langle \mathbf X_u, \rho_{u\to v} \rangle_{\boldsymbol\kappa}$ with $\boldsymbol\kappa={\langle \kappa_o, \kappa_J, \kappa_I, \leq_\kappa\rangle}$,
$$\mathcal G_{\mathcal A}=\langle \textbf{\textit{G$_u$}},\textbf{\textit{H$_u$}}, \varsigma_{u\to v} \rangle_{\boldsymbol\kappa}$$
is bunch of layer groups, where
\begin{equation}\label{DEFcsopi}
\textbf{\textit{G$_u$}}=(G_u,\preceq_u,\cdot_u,\ { }^{-1_u},u)=
\left\{
\begin{array}{ll}
\mbox{$\lambda(\mathbf X_u)$} & \mbox{if $u\in\kappa_o$}\\
\mbox{$\lambda\left({\mathbf X_u}_{\contour{black}{$_\uparrow$}}\right)$} & \mbox{if $u\in\kappa_J$}\\
\mbox{$\lambda(\pi_1(\mathbf X_u))$
}
& \mbox{if $u\in\kappa_I$}\\
\end{array}
\right. ,
\end{equation}
for $u\in\kappa_I$, 
\begin{equation}\label{DEFcsopiH}
\textbf{\textit{H$_u$}}=(H_u,\preceq_u,\cdot_u,\ { }^{-1_u},u)=\lambda(\pi_2(\mathbf X_u)),
\end{equation}
$\kappa=\kappa_o\cup\kappa_J\cup\kappa_I$, 
and for $u,v\in\kappa$ such that $u\leq_\kappa v$, $\varsigma_{u\to v} : G_u\to G_v$ is defined by
\begin{equation}\label{DEFvarsigma}
\varsigma_{u\to v}=
\rho_{u\to v}|_{G_u}
.
\end{equation}
Call $\mathcal G_{\mathcal A}$ the {\em bunch of layer groups derived from $\mathcal A$}.
\item 
Given a bunch of layer groups $\mathcal G=\langle \textbf{\textit{G$_u$}},\textbf{\textit{H$_u$}}, \varsigma_{u\to v} \rangle_{\boldsymbol\kappa}$  with $\boldsymbol\kappa={\langle \kappa_o, \kappa_J, \kappa_I, \leq_\kappa\rangle}$,
$${\mathcal A}_{\mathcal G}=\langle \mathbf X_u, \rho_{u\to v} \rangle_{\boldsymbol\kappa}$$
is bunch of layer algebras, 
called the {\em bunch of layer algebras derived from $\mathcal G$}, 
where
\begin{equation}\label{GroupToAlgebra}
\mathbf X_u=(X_u,\leq_u,\teu,\ite{u},u,\negaM{u}{u})
=
\left\{
\begin{array}{ll}
\iota(\textbf{\textit{G$_u$}}) 
& \mbox{if $u\in\kappa_o$}\\
\mbox{$\iota(\textbf{\textit{G$_u$}})_{\contour{black}{$_\downarrow$}}$}
& \mbox{if $u\in\kappa_J$}\\
Sp(\iota(\textbf{\textit{G$_u$}}),\iota(\textbf{\textit{H$_u$}})),
& \mbox{if $u\in\kappa_I$}\\
\end{array}
\right. ,
\end{equation}
$\kappa=\kappa_o\cup\kappa_J\cup\kappa_I$,
and for $u,v\in\kappa$ such that $u\leq_\kappa v$, $\rho_{u\to v} : X_u\to X_v$ is defined by
\begin{equation}\label{DEFrho}
\rho_{u\to v}=
\left\{
\begin{array}{ll}
\varsigma_{u\to v} & \mbox{if $u\notin\kappa_I$}\\
\varsigma_{u\to v}\circ h_u & \mbox{if $v>u\in\kappa_I$}\\
id_{X_u} & \mbox{if $v=u\in\kappa_I$}\\
\end{array}
\right. ,
\end{equation}
where $h_u$ is the canonical homomorphism of $\mathbf X_u$. 
\item 
Given a bunch of layer groups $\mathcal G$, it holds true that 
$\mathcal G_{({\mathcal A}_\mathcal G)}=\mathcal G$, and
given a bunch of layer algebras $\mathcal A$, it holds true that 
$\mathcal A_{({\mathcal G}_\mathcal A)}=\mathcal A$.
\end{enumerate}
\end{lemma}
\begin{proof}
(1):
For $u\in\kappa$, \textbf{\textit{G$_u$}} defined by (\ref{DEFcsopi}) is an abelian $o$-group.
Indeed, being totally ordered is granted since so is the original algebra. 
If $u\in\kappa_o$ then see (\ref{KiKiLesz}) and Lemma~\ref{inducedDEF}, 
if $u\in\kappa_J$ then see (\ref{KiKiLesz}) and Lemmas~\ref{PROdownshift} and \ref{inducedDEF}, it also confirms (\ref{DiSCRetE}), whereas if $u\in\kappa_I$ then see  (\ref{KiKiLesz}), Theorem~\ref{sPliT} and Lemma~\ref{inducedDEF}. 
The $\varsigma$'s defined in (\ref{DEFvarsigma}) satisfy
\begin{itemize}
\item[(G1):] 
for $x\in G_u$,
$
\varsigma_{u\to u}(x)
\overset{(\ref{DEFvarsigma})}{=}
\rho_{u\to u}(x)
\overset{(A1)}{=}
x
$
and
$
(\varsigma_{v\to w}\circ\varsigma_{u\to v})(x)
\overset{(\ref{DEFvarsigma})}{=}
(\rho_{v\to w}\circ\rho_{u\to v})(x)
\overset{\rm{(A1)}}{=}
\rho_{u\to w}(x)
\overset{(\ref{DEFvarsigma})}{=}
\varsigma_{u\to w}(x)
$.

\item[(G2):] 
Let $u\in\kappa_J$. 
Then
$
\varsigma_{u\to v}(u_{\downarrow_u})
\overset{(\ref{DEFvarsigma})}{=}
\rho_{u\to v}(u_{\downarrow_u})
\overset{(\ref{KiKiLesz})}{=}
\rho_{u\to v}(\negaM{u}{u})
\overset{\rm{(A2)}}{=}
\rho_{u\to v}(u)
\overset{(\ref{DEFvarsigma})}{=}
\varsigma_{u\to v}(u)
$.

\item[(G3):] 
Let $v\in\kappa_I$.
Then $\mathbf X_v=Sp(\pi_1(\mathbf X_v),\pi_2(\mathbf X_v))$ holds by Theorem~\ref{sPliT}.
For $x\in G_u$, 
$\varsigma_{u\to v}(x)
\overset{(\ref{DEFvarsigma})}{=}
\rho_{u\to v}(x)
\in X_v
$, and
by claim~(C1) in the proof of claim~(2) of Lemma~\ref{BUNCHalg_X}, $\rho_{u\to v}(x)$ is invertible in $\mathbf X_v$.
Therefore, $\rho_{u\to v}(x)$ is an element of
$
\pi_2(\mathbf X_v)
\overset{(\ref{DEFcsopiH})}{=}
\mathbf H_u
$
by claim~(\ref{UNICITYofH}) in Theorem~\ref{sPliT}.
\end{itemize}

\medskip
(2):
$\mathbf X_u$ defined in (\ref{GroupToAlgebra}) 
is an involutive FL$_e$-chain satisfying (\ref{KiKiLesz}):
if $u\in\kappa_o$ then see Lemma~\ref{inducedDEF}, 
if $u\in\kappa_J$ then see Lemmas~\ref{inducedDEF} and \ref{PROdownshift}, 
if $u\in\kappa_I$ then see Lemma~\ref{inducedDEF} and Theorem~\ref{sPliT}.
\\
The $\rho$'s defined in (\ref{DEFrho}) are well defined since $X_u=G_u$ holds by (\ref{GroupToAlgebra}) if $u\notin\kappa_I$,
and if $u\in\kappa_I$ then $h_u$ maps to the universe of $\iota(\textbf{\textit{G$_u$}})$ (see (\ref{canHOMmegint})), which is 
$G_u$. 
\\
The $\rho$'s 
are residuated lattice homomorphisms,
since the (totally ordered group) homomorphisms $\varsigma_{u\to v}$ from 
$\textbf{\textit{G$_u$}}$ to $\textbf{\textit{G$_v$}}$
naturally extend to homomorphisms from the residuated lattice reduct of
$\iota(\textbf{\textit{G$_u$}})$ to 
the residuated lattice reduct of $\iota(\textbf{\textit{G$_v$}})$
via claim~(\ref{CsopToAlg}) of Lemma~\ref{EzAzInverz}, and hence $\rho_{u\to v}$ can be regarded as the composition of residuated lattice homomorphisms ($\varsigma$'s and $h$'s).
\\
The $\rho$'s 
satisfy 

(A1):
Notice that for $u<v$,
\begin{equation}\label{HOVArho}
\mbox{$\rho_{u\to v}$ maps $X_u$ to $G_v$,}
\end{equation}
since so does $\varsigma_{u\to v}$.
Over $X_u$, 
$$
\rho_{u\to u}
\overset{(\ref{DEFrho})}{=}
\left\{
\begin{array}{ll}
\varsigma_{u\to u}\overset{(G1)}{=}
id_{G_u}
\overset{(\ref{GroupToAlgebra})}{=}
id_{X_u} & \mbox{if $u\notin\kappa_I$}\\
id_{X_u} & \mbox{if $v=u\in\kappa_I$}\\
\end{array}
\right. .
$$
Therefore, it suffices to prove the other condition in (A1) for $u<v<w$ only:
\\
$
\rho_{v\to w}\circ\rho_{u\to v}
\overset{(\ref{DEFrho})}{=}
\left\{
\begin{array}{ll}
\varsigma_{v\to w}\circ\rho_{u\to v} & \mbox{if $v\notin\kappa_I$}\\
\varsigma_{v\to w}\circ h_v\circ\rho_{u\to v}
\overset{(\ref{HOVArho}) \, (\ref{canHOM})}{=}
\varsigma_{v\to w}\circ\rho_{u\to v}
& \mbox{if $v\in\kappa_I$}\\
\end{array}
\right\}
=
\varsigma_{v\to w}\circ\rho_{u\to v}
\overset{(\ref{DEFrho})}{=}
\left\{
\begin{array}{ll}
\varsigma_{v\to w}\circ\varsigma_{u\to v}
\overset{\rm{(G1)}}{=}
\varsigma_{u\to w}
& \mbox{if $u\notin\kappa_I$}\\
\varsigma_{v\to w}\circ\varsigma_{u\to v}\circ h_u 
\overset{\rm{(G1)}}{=}
\varsigma_{u\to w}\circ h_u 
& \mbox{if $u\in\kappa_I$}\\
\end{array}
\right\}
\overset{(\ref{DEFrho})}{=}
\rho_{u\to w}
.
$

(A2):
For $u\notin\kappa_o$ we have already seen that 
\begin{equation}\label{EveNNN}
\mathbf X_u \mbox{ is even} 
,
\end{equation}
therefore,
$
\rho_{u\to v}(\negaM{u}{u})
\overset{(\ref{DEFrho})}{=}
$
$$
\footnotesize
\left\{
\begin{array}{ll}
\varsigma_{u\to v}(\negaM{u}{u})
\overset{(\ref{EveNNN})}{=}
\varsigma_{u\to v}(u_{\downarrow_u})\overset{\rm{(G2)}}{=}
\varsigma_{u\to v}(u)
\overset{(\ref{DEFrho})}{=}
\rho_{u\to v}(u)
& \mbox{if $u\in\kappa_J$,}\\
(\varsigma_{u\to v}\circ h_u)(\negaM{u}{u})
\overset{(\ref{canHOM})}{=}
\varsigma_{u\to v}(\negaM{u}{u}_\uparrow)
\overset{(\ref{EveNNN})}{=}
\varsigma_{u\to v}(u)
\overset{(\ref{canHOM})}{=}
(\varsigma_{u\to v}\circ h_u)(u)
\overset{(\ref{DEFrho})}{=}
\rho_{u\to v}(u)
& \mbox{if $u\in\kappa_I$.}\\
\end{array}
\right.
$$

\bigskip
(3):
If $u=\kappa_o$ then $\lambda(\iota(\textbf{\textit{G$_u$}}))=\textbf{\textit{G$_u$}}$ and $\iota(\lambda(\mathbf X_u))=\mathbf X_u$ by Lemma~\ref{inducedDEF}.
If $u=\kappa_J$ then $\lambda\left({\iota(\textbf{\textit{G$_u$}})_{\contour{black}{$_\downarrow$}{\contour{black}{$_\uparrow$}}}}\right)=\lambda(\iota(\textbf{\textit{G$_u$}}))=\textbf{\textit{G$_u$}}$
and
$
\iota\left(\lambda\left({\mathbf X_u}_{\contour{black}{$_\uparrow$}}\right)\right)_{\contour{black}{$_\downarrow$}}
=
{\mathbf X_u}_{\contour{black}{$_\uparrow$}{\contour{black}{$_\downarrow$}}}=\mathbf X_u
$
follow from Lemma~\ref{PROdownshift} and Lemma~\ref{inducedDEF}.
In these two cases $G_u=X_u$, thus it is obvious from (\ref{DEFrho}) and (\ref{DEFvarsigma}) that 
$\mathcal G_{({\mathcal A}_\mathcal G)}$ and $\mathcal G$ have the same homomorphisms (the same $\varsigma$'s) from the $u^{\rm th}$-layer, and that 
$\mathcal A_{({\mathcal G}_\mathcal A)}$ and $\mathcal A$ have the same homomorphisms (the same $\rho$'s) from the $u^{\rm th}$-layer.
If $u\in\kappa_I$ then
$
\lambda(\pi_1(Sp(\iota(\textbf{\textit{G$_u$}}),\iota(\textbf{\textit{H$_u$}}))))
=
\lambda(\iota(\textbf{\textit{G$_u$}}))
=
\textbf{\textit{G$_u$}}
,
$
$
\lambda(\pi_2(Sp(\iota(\textbf{\textit{G$_u$}}),\iota(\textbf{\textit{H$_u$}}))))
=
\lambda(\iota(\textbf{\textit{H$_u$}}))
=
\textbf{\textit{H$_u$}}
$
and
$
Sp(\iota(\lambda(\pi_1(\mathbf X_u))),\iota(\lambda(\pi_2(\mathbf X_u))))
=
Sp(\pi_1(\mathbf X_u),\pi_2(\mathbf X_u))
=
\mathbf X_u
$
follow from Theorem~\ref{sPliT} and Lemma~\ref{inducedDEF}.
As for the homomorphisms, $h_u$ maps $X_u$ to $G_u$ by (\ref{DEFcsopi}),
hence the composition $\varsigma_{u\to v}\circ h_u$ is well-defined.
By the construction in Definition~\ref{SubgroupSplitREV},
$G_u$ can also be regarded as a subset of $X_u$, and $h_u$ is the identity mapping on $G_u$ by (\ref{canHOMmegint}).
Therefore, 
$(\varsigma_{u\to v}\circ h_u)|_{G_u}
=
\varsigma_{u\to v}
$
holds on the one hand.
On the other hand, to prove $\rho_{u\to v}|_{G_u}\circ h_u=\rho_{u\to v}$, first notice that 
for $x\in X_u\setminus G_u$,
\begin{equation}\label{RHOugyanodaVISZI}
\rho_{u\to v}(x_{\uparrow_u})=\rho_{u\to v}(x)
.
\end{equation}
Indeed, 
$
\rho_{u\to v}(x_{\uparrow_u})
=
\rho_{u\to v}(\gteu{x_{\uparrow_u}}{u})
\overset{(\ref{RHOhomo})}{=}
\gtev{\rho_{u\to v}(x_{\uparrow_u})}{\rho_{u\to v}(u)}
\overset{(A2)}{=}
\gtev{\rho_{u\to v}(x_{\uparrow_u})}{\rho_{u\to v}(\negaM{u}{u})}
\overset{X_u\ is\ even,\ see\ (\ref{KiKiLesz})}{=}
\gtev{\rho_{u\to v}(x_{\uparrow_u})}{\rho_{u\to v}(u_{\downarrow_u})}
\overset{(\ref{RHOhomo})}{=}
\rho_{u\to v}(\gteu{x_{\uparrow_u}}{u_{\downarrow_u}})
\overset{(\ref{ProdSplit})}{=}
\rho_{u\to v}(x_{{\uparrow_u}{\downarrow_u}})
=
\rho_{u\to v}(x)
$.
Therefore, 
$
(\rho_{u\to v}|_{G_u}\circ h_u)(x)
\overset{(\ref{canHOM})}{=}
$
$$
\overset{(\ref{canHOM})}{=}
\left\{
\begin{array}{ll}
(\rho_{u\to v}|_{G_u})(x)
=
\rho_{u\to v}(x)
& \mbox{ if $x\in G_u$}\\
(\rho_{u\to v}|_{G_u})(x_{\uparrow_u})
\overset{x_{\uparrow_u}\in H_u\subseteq G_u}{=}
\rho_{u\to v}(x_{\uparrow_u})
\overset{(\ref{RHOugyanodaVISZI})}{=}
\rho_{u\to v}(x)
& \mbox{ if $x\in X_u\setminus G_u$}\\
\end{array}
\right. .
$$
\end{proof}

\section{The representation theorem}\label{NaVegre}
The main theorem of the paper is a representation theorem of odd or even involutive FL$_e$-chains by bunches of layer groups. 
Lemmas~\ref{BUNCHalg_X} and \ref{BUNCHalg_BUNCHgroup} prove
Theorem~\ref{mainTheoREm}. The second part of Theorem~\ref{mainTheoREm} presents the direct constructional correspondence between odd or even involutive FL$_e$-chains and bunches of layer-groups, that is, one without referring to the intermediate explanatory step of layer algebras.

\begin{theorem}\label{mainTheoREm}
For every odd or even involutive FL$_e$-chain $\mathbf X$ there exists a unique bunch of layer groups $\mathcal G$ such that 
$\mathbf X$ is the involutive FL$_e$-chain derived from the bunch of layer algebras derived from $\mathcal G$, in notation, $\mathbf X=\mathcal X_{{\mathcal A}_{\mathcal G}}$.
Conversely, for every bunch of layer groups $\mathcal G$ there exists a unique odd or even involutive FL$_e$-chain $\mathbf X$ such that $\mathcal G$ is the bunch of layer groups derived from the bunch of layer algebras of $\mathbf X$, in notation, $\mathcal G=\mathcal G_{{\mathcal A}_{\mathbf X}}$. 
In more details:
\begin{enumerate}[label={\bf (A)}]
\item\label{errefere}
Given an odd or an even involutive FL$_e$-chain $\mathbf X=(X,\leq,\teALONE,\ite{\te},t,f)$ with residual complement operation $\komp$,
$$
\mathcal G_{\mathbf X}=\langle \textbf{\textit{G$_u$}},\textbf{\textit{H$_u$}}, \varsigma_{u\to v} \rangle_{\langle \kappa_o, \kappa_J, \kappa_I,\leq_\kappa\rangle}
$$
is bunch of layer groups, called the {\em bunch of layer groups of $\mathbf X$},
where
$$
\kappa=\{\res{\te}{x}{x} : x\in X\}=
\{u\geq t : u \mbox{ is idempotent} \} 
\mbox{ is ordered by $\leq$,}
$$
 $$
\bar\kappa_I=\{u\in \kappa\setminus\{t\} : \nega{u} \mbox{ is idempotent}\},
$$
$$
\bar\kappa_J=\{u\in \kappa\setminus\{t\} : \nega{u} \mbox{ is not idempotent}\},
$$
$\kappa_o$, $\kappa_J$, $\kappa_I$ are defined by Table~\ref{ThetaPsiOmegaAGAIN},
\begin{table}[h]
\begin{center}
\caption{}
\label{ThetaPsiOmegaAGAIN}
\begin{tabular}{c|c|c|cll}
$\kappa_o$ & $\kappa_J$ & $\kappa_I$ & \\
\hline
\{t\} & $\bar\kappa_J$ & $\bar\kappa_I$ & if $\mathbf X$ is odd\\
\hline
$\emptyset$ & $\bar\kappa_J\cup\{t\}$ & $\bar\kappa_I$ & if $\mathbf X$ is even and $f$ is not idempotent \\
\hline
$\emptyset$ & $\bar\kappa_J$ & $\bar\kappa_I\cup\{t\}$ & if $\mathbf X$ is even and $f$ is idempotent\\
\hline
\end{tabular}
\end{center}
\end{table}

\noindent
$$
\begin{array}{llll}
\textbf{\textit{G$_u$}}&=& (G_u,\leq,\teALONE,\ { }^{-1},u) & \mbox{if $u\notin\kappa_I$,}\\
\textbf{\textit{G$_u$}}&=& (G_u,\leq,\teALONE_u,\ { }^{-1},u) & \mbox{if $u\in\kappa_I$,}\\
\textbf{\textit{H$_u$}}&=& (H_u,\leq,\teALONE,\ { }^{-1},u) & \mbox{if $u\in\kappa_I$,}\\
\end{array}
$$
where 
$X_u=\{x\in X : \res{\te}{x}{x}=u\}$, 
$H_u=\{x\in X_u : \g{x}{\nega{u}}<x\}=
\{x\in X_u : x \mbox{ is $u$-invertible in $X_u$}\footnote{There exists $y\in X_u$ such that $\g{x}{y}=u$.}\},
$,
$\accentset{\bullet}H_u=\{\g{x}{\nega{u}} : x\in H_u\}$,
\begin{equation*}\label{DEFcsopi2}
G_u=\left\{
\begin{array}{ll}
X_u & \mbox{if $u\notin\kappa_I$}\\
X_u\setminus H_u^\bullet & \mbox{if $u\in\kappa_I$}\\
\end{array}
\right.
\end{equation*}
$$
 x\mathbin\teALONE_u y=\res{\te}{(\res{\te}{\g{x}{y}}{u})}{u}
$$
\begin{equation*}\label{EzLeSzainVerZ}
x^{-1}=\res{\te}{x}{u},
\end{equation*} 

and for $u,v\in\kappa$ such that $u\leq v$, $\varsigma_{u\to v} : G_u\to G_v$ is defined by
\begin{equation*}\label{IgYszoRZUnk}
\varsigma_{u\to v}(x)=
\g{v}{x}
.
\end{equation*}
\end{enumerate}

\bigskip
\begin{enumerate}[label={\bf (B)}]
\item\label{BLGtoX}

Given a bunch of layer groups
$
\mathcal G=\langle \textbf{\textit{G$_u$}},\textbf{\textit{H$_u$}}, \varsigma_{u\to v} \rangle_{{\langle \kappa_o, \kappa_J, \kappa_I, \leq_\kappa\rangle}}
$
with $\textbf{\textit{G$_u$}}=(G_u,\preceq_u,\cdot_u,\ { }^{-1_u},u)$
$$
\mathbf X_{\mathcal G}=(X,\leq,\teALONE,\ite{\te},t,\nega{t})
$$
is an involutive FL$_e$-chain with residual complement $\komp$,
called the {\em involutive FL$_e$-chain of $\mathcal X$},
where
$\kappa=\kappa_o\cup\kappa_J\cup\kappa_I$,
for $u\in\kappa$,
\begin{equation*}\label{IkszU}
X_u=\left\{
\begin{array}{ll}
G_u & \mbox{ if $u\not\in\kappa_I$},\\
G_u\,
\cup\, H_u^\bullet & \mbox{ if $u\in\kappa_I$},\\
\end{array}
\right. 
\end{equation*}
(where $H_u^\bullet=\{{h^\bullet} : h\in H_u\}$ is a copy of $H_u$ which is disjoint from $G_u$), \begin{equation*}\label{EZazX2}
X=\displaystyle\dot\bigcup_{u\in \kappa}X_u
,
\end{equation*}
if $u\notin\kappa_I$ then $\leq_u\,=\,\preceq_u$, 
if $u\in\kappa_I$ then 
$\leq_u$ extends $\preceq_u$ to $X_u$ 
by letting 
\begin{equation*}\label{KibovitettRendezesITTIS}
\mbox{
$a^\bullet<_u b^\bullet$ and $x<_u a^\bullet<_uy$
if $a,b\in H_u$, $x,y\in G_u$, $a\prec_u b$, $x\prec_u a\preceq_u y$,
}
\end{equation*}
for $v\in\kappa$, $\rho_v : X\to X$ is defined by
\begin{equation*}\label{P52}
\begin{array}{lll}
\rho_v(x)&=&
\left\{
\begin{array}{ll}
\varsigma_{u\to v}(x) & \mbox{ if $x\in G_u$ and $u<_\kappa v$},\\
x & \mbox{ if $x\in G_u$ and $u\geq_\kappa v$},\\
\end{array}
\right. 
\\
\rho_v(\accentset{\bullet}x)&=&
\left\{
\begin{array}{ll}
\varsigma_{u\to v}(x) & \mbox{ if $x^\bullet\in H_u^\bullet$ and $\kappa_I\ni u<_\kappa v$},\\
x^\bullet & \mbox{ if $x^\bullet\in H_u^\bullet$ and $\kappa_I\ni u\geq_\kappa v$},\\
\end{array}
\right. 
\end{array}
\end{equation*}
by denoting for $u,v\in\kappa$, $uv=\max_\kappa(u,v)$,
for $x\in X_u$ and $y\in X_v$,
\begin{equation*}\label{RendeZesINNOVATIVAN}
\mbox{$x<y$ iff $\rho_{uv}(x)<_{uv}\rho_{uv}(y)$ or $\rho_{uv}(x)=\rho_{uv}(y)$ and $u<_\kappa v$,}\\
\end{equation*}
for $u\in\kappa_I$, $h_u : X_u \to G_u$, 
\begin{equation*}\label{canHOM2}
\begin{array}{ll}
h_u(x)=
x & \mbox{ if $x\in G_u$,}\\
h_u(x^\bullet)=
x & \mbox{ if $x^\bullet\in H_u^\bullet$,}\\
\end{array}
\end{equation*}
for $x,y\in X_u$,
\begin{equation*}\label{uPRODigy}
{x}\mathbin{\bigcdot_u}{y}=\left\{
\begin{array}{ll}
{\left({h_u(x)}\cdot_u{h_u(y)}\right)}^\bullet	& \mbox{ if $u\in\kappa_I$, $\gteu{h_u(x)}{h_u(y)}\in H_u$ and $\neg(x,y\in H_u)$}\\
{h_u(x)}\cdot_u{h_u(y)}		& \mbox{ if $u\in\kappa_I$, $\gteu{h_u(x)}{h_u(y)}\notin H_u$ or $x,y\in H_u$}\\
x\cdot_u y& \mbox{ if $u\notin\kappa_I$}\\
\end{array}
\right. ,
\end{equation*}
for $x\in X_u$ and $y\in X_v$,
\begin{equation*}\label{EgySzeruTe2}
\g{x}{y}={\rho_{uv}(x)}\mathbin{\bigcdot_{uv}}{\rho_{uv}(y)},
\end{equation*}
for $x\in X$,
\begin{equation*}\label{SplitNega2}
\begin{array}{lll}
\nega{(x^\bullet)}&=&\left\{ 
\begin{array}{ll}
x^{-1_u}		& \mbox{ \ \ \ if $u\in\kappa_I$ and $x^\bullet\in H_u^\bullet$}\\
\end{array}
\right. \\
\nega{x}&=&\left\{
\begin{array}{ll}
{\left(x^{-1_u}\right)}^\bullet	& \mbox{ if $u\in\kappa_I$ and $x\in H_u$}\\
x^{-1_u}	& \mbox{ if $u\in\kappa_I$ and $x\in G_u\setminus H_u$}\\
x^{-1_u}		& \mbox{ if $u\in\kappa_o$ and $x\in G_u$}\\
{x^{-1_u}}_\downarrow		& \mbox{ if $u\in\kappa_J$ and $x\in G_u$}\\
\end{array}
\right. ,
\end{array}
\end{equation*}
for $x,y\in X$,
\begin{equation*}\label{IgYaReSi2}
\res{\te}{x}{y}=\nega{(\g{x}{\nega{y}})},
\end{equation*}
$\ite{\te}$ is the residual operation of $\teALONE$,
$$
\mbox{
$t$ is the least element of $\kappa$,
}
$$
\begin{equation*}\label{tLESZaz}
\mbox{
$f$ is the residual complement of $t$,
}
\end{equation*}
and is given by
$$
\begin{array}{lll}
\nega{t}&=&\left\{
\begin{array}{ll}
{(t^{-1_t})}^\bullet	& \mbox{ if $u\in\kappa_I$}\\
t^{-1_t}		& \mbox{ if $u\in\kappa_o$}\\
{t^{-1_t}}_\downarrow		& \mbox{ if $u\in\kappa_J$}\\
\end{array}
\right. .
\end{array}
$$
In addition, 
\begin{equation*}\label{RhOLESzEZwww}
\rho_v(x)=\g{v}{x} \ \ \mbox{for $v\in\kappa$ and $x\in X$},
\end{equation*}
$\mathbf X_{\mathcal X}$ is odd if $t\in\kappa_o$, 
even with a non-idempotent falsum if $t\in\kappa_J$, and 
even with an idempotent falsum if $t\in\kappa_I$.
\end{enumerate}

\begin{enumerate}[start=1,label={\bf (C)}]
\item
Items~\ref{errefere} and {\bf\ref{BLGtoX}} describe a one-to-one correspondence between the class containing all odd and all even involutive FL$_e$-chains and the class of bunches of layer groups:
given a bunch of layer groups $\mathcal X$ it holds true that 
$\mathcal X_{({\mathbf X}_\mathcal X)}=\mathcal X$, and
given an odd or even involutive FL$_e$-chain $\mathbf X$ it holds true that $\mathbf X_{(\mathcal X_\mathbf X)}\simeq\mathbf X$\footnote{\label{ModifiCaTO}
If Definition~\ref{DEFbunch} is slightly modified in such a way that the
$\accentset{\bullet}{\textbf{\textit{H$_u$}}}$'s
are \lq\lq stored\rq\rq\ in the definition of a bunch
(like ${\mathbf X}=\langle \textbf{\textit{G$_u$}},\textbf{\textit{H$_u$}},\accentset{\bullet}{\textbf{\textit{H$_u$}}}, \varsigma_{u\to v} \rangle_{\langle \kappa_o, \kappa_J, \kappa_I, \leq_\kappa\rangle}$),
and 
instead of taking a copy $\accentset{\bullet}{H}_u$ of $H_u$, that stored copy is used in the construction of Theorem~\ref{mainTheoREm}/\ref{errefere},
then also
$\mathbf X_{(\mathcal X_\mathbf X)}=\mathbf X$
holds.
Then, Theorem~\ref{mainTheoREm} describes a bijection, in a constructive manner, between the classes of odd or even involutive FL$_e$-chains and the class of bunches of layer groups.}.
\qed
\end{enumerate}
\end{theorem}

\begin{example}
We present the bunch representation of a few known structures, among which are the two extremal classes 
(abelian $o$-groups and odd Sugihara chains)
mentioned in the introduction. 
Denote by $\mathbbm 1$ the trivial (one-element) group.
\begin{itemize}
\item 
If \textbf{\textit{G}} is an abelian $o$-group then
$\textbf{\textit{G}}=\mathbf X_{\mathcal G}$
where
$${\mathcal G}=\langle \textbf{\textit{G}}, \emptyset, \emptyset \rangle_{\langle\{t\},\emptyset,\emptyset,\leq_\kappa\rangle}.$$
\item 
Even Sugihara chains are exactly the algebras
$\mathbf X_{\mathcal G}$,
where 
$${\mathcal G}=\langle \mathbbm 1_u, \mathbbm 1_u, \varsigma_{u\to v} \rangle_{\langle\emptyset,\emptyset,\kappa,\leq_\kappa\rangle}.$$
\item 
Odd Sugihara chains are exactly the algebras
$\mathbf X_{\mathcal G}$,
where 
$${\mathcal G}=\langle \mathbbm 1_u, \mathbbm 1_u, \varsigma_{u\to v} \rangle_{\langle\{t\},\emptyset,\kappa\setminus\{t\},\leq_\kappa\rangle}.$$
\item
Finite partial sublex products of abelian $o$-groups 
have been shown in \cite{JS_Hahn,Jenei_Hahn_err} to be exactly those odd involutive FL$_e$-chains which have finitely many positive idempotent elements. These are exactly the algebras $\mathbf X_{\mathcal G}$, where $\kappa$ is finite in 
$$
{\mathcal G}=\langle \textbf{\textit{G$_u$}},\textbf{\textit{H$_u$}}, \varsigma_{u\to v} \rangle_{\langle\{t\},\bar\kappa_J,\bar\kappa_I,\leq_\kappa\rangle}.
$$

\item
Algebras which can be constructed by the involutive ordinal sum construction of \cite{InvOrdSum} are exactly the algebras $\mathbf X_{\mathcal G}$, where 
$$
{\mathcal G}=\langle \textbf{\textit{G$_u$}}, \mathbbm 1_u, \varsigma_{u\to v} \rangle_{\langle\{t\},\emptyset,\kappa\setminus\{t\},\leq_\kappa\rangle}.$$
By (G3), also the homomorphisms are trivial.
\item
Algebras which can be constructed by the consecutive application of the ordinal sum construction as defined in  \cite{Galatos:2005kh} are exactly the algebras $\mathbf X_{\mathcal G}$, where $\kappa$ is finite 
in 
$${\mathcal G}=\langle \textbf{\textit{G$_u$}}, \mathbbm 1_u, \varsigma_0^{u\to v} \rangle_{\langle\{t\},\emptyset,\kappa\setminus\{t\},\leq_\kappa\rangle}.$$
By (G3), also the homomorphisms are trivial.
\end{itemize}
\end{example}

\begin{remark}
If $\mathbf X$ is densely ordered in Theorem~\ref{mainTheoREm} then the \textbf{\textit{H$_u$}}'s in the  representation 
${\mathcal G}=\langle \textbf{\textit{G$_u$}},\textbf{\textit{H$_u$}}, \varsigma_{u\to v} \rangle_{\boldsymbol\kappa}$ of $\mathbf X$ are uniquely determined by the rest of $\mathcal G$.
Therefore, if $\mathbf X$ is densely ordered then the representation of $\mathbf X$ by layer groups can be written in a simpler form of
$\langle \textbf{\textit{G$_u$}}, \varsigma_{u\to v} \rangle_{\langle \kappa_o, \kappa_J, \kappa_I,\leq_\kappa\rangle}$.
To prove it we state that for $u\in\kappa_I$, 
$$
H_u=\displaystyle\bigcup_{\kappa\ni s<_\kappa u}\varsigma_{s\to u}(G_s).
$$
Indeed, 
$
\textit{H$_u$}\supseteq\bigcup_{\kappa\ni s<_\kappa u}\varsigma_{s\to u}(\textit{G$_s$})
$
follows from (G3). 
Since $u\in\kappa_I$, and since $\boldsymbol\kappa$ is the same in $\mathcal G_{{\mathcal A}_\mathbf X}$ and in ${\mathcal A}_\mathbf X$, 
the $u^{\rm th}$-layer algebra $\mathbf X_u$ of $\mathbf X$ has idempotent falsum constant $\nega{u}$ by (\ref{KiKiLesz}).
Therefore, by Theorem~\ref{sPliT}, $\mathbf X_u=Sp(\pi_1(\mathbf X_u), \pi_2(\mathbf X_u))$, where $\pi_2(\mathbf X_u)=\textbf{\textit{H$_u$}}$ comprises the following elements $H_u=\{x\in X_u : \g{x}{\nega{u}}<x\}=\{x\in X : \tau(x)=u, \g{x}{\nega{u}}<x\}$. 
Hence proving 
$
\textit{H$_u$}\subseteq\bigcup_{\kappa\ni s<_\kappa u}\varsigma_{s\to u}(\textit{G$_s$})
$
amounts to showing that
in every densely ordered, odd or even involutive FL$_e$-chain $\mathbf X=(X,\leq,\teALONE,\ite{\te},t,f)$ with residual complement operation $\komp$, 
if $u\geq t$ and $\nega{u}$ are idempotent, and $x\in X$ such that $\tau(x)=u$ and $\g{x}{\nega{u}}<x$ then 
there exist a positive idempotent element $X\ni s<u$ and $y\in X$ such that $\tau(y)=s$ and $\g{y}{u}=x$.
Let $\g{x}{\nega{u}}<y<x$ (such $y$ exists since $X$ is densely ordered).
Now, 
$
\nega{y}<\nega{(\g{x}{\nega{u}})}\leq\g{\nega{y}}{u}
$
follows by claim~\ref{eq_feltukrozes_CS} in Lemma~\ref{tuttiINVOLUTIVE}.
Therefore, $u\not\in Stab_{\nega{y}}$ and $s:=\tau(y)=\tau(\nega{y})<u$ follows.
Since 
$
\g{y}{u}
\overset{(\ref{EzARho})}{=}
\rho_{s\to u}(y)
\overset{(\ref{P5})}{=}
\rho_u(y)
\overset{(\ref{HgHJJhGkhj})}{=}
\min\{z\in X_u: z\geq y\}
$
and since
$X_u\ni\g{x}{\nega{u}}<y<x\in X_u$,
to see that $\g{y}{u}=x$,
it suffices to prove that
$
\g{x}{\nega{u}}
=
x_{\downarrow_u}
$\footnote{Recall that $x_{\downarrow_u}$ is computed in $X_u$.}.
But it holds true since $x\in H_u$ and $\nega{u}\in  H_u^\bullet$, and hence
$
\g{x}{\nega{u}}
=
\g{x}{u_{\downarrow_u}}
=
\g{x}{u^\bullet}
\overset{Table~\ref{Reszeik}/\star_{(2,3)}}{=}
x^\bullet
=
x_{\downarrow_u}
$
follows from claim~\ref{AlaLovokk} in the proof of Theorem~\ref{sPliT}/(\ref{VisszIsMEGyeget}), and we are done.
\end{remark}

\begin{remark}
The easiest way of generalizing Theorem~\ref{mainTheoREm} to conic algebras is to observe that there exist no conic odd or even involutive FL$_e$-algebras which are not linearly ordered.
Indeed, let $a,b\in X$. Proving that $a$ and $b$ are comparable, that is, $a\leq b$ or $b\leq a$ amounts to proving $\g{a}{t}\leq b$ or $\g{b}{t}\leq a$, or equivalently, $\res{\te}{a}{b}\geq t$ or $\res{\te}{b}{a}\geq t$ by adjointness. 
If $\res{\te}{a}{b}\not\geq t$ then $\res{\te}{a}{b}<t$ since $X$ is conic, hence $\res{\te}{a}{b}\leq f$ since $X$ is odd or even.
By claims~\ref{MiNdig} and \ref{eq_quasi_inverse} in Lemma~\ref{tuttiINVOLUTIVE}, $\g{a}{\nega{b}}\geq t$ follows.
Therefore, $\nega{(\g{\nega{a}}{b})}\geq t$ holds by claim~\ref{eq_feltukrozes} in Lemma~\ref{tuttiINVOLUTIVE},
hence $\res{\te}{b}{a}\geq t$ holds by claim~\ref{eq_quasi_inverse} in Lemma~\ref{tuttiINVOLUTIVE}.
\end{remark}


\begin{thebibliography}{lll}

\bibitem{AglMon} P. Agliano, F. Montagna: Varieties of BL-algebras I: general properties, Journal of Pure and Applied Algebra, 181 vol. 2-3 (2003), 105-129.
\bibitem{AB} A.R. Anderson, N.D. Belnap, {Entailment. Volume I: The logic of relevance and necessity.}, Princeton University Press, Princeton, N. J., 1975.
\bibitem{AF} 
M. Anderson, T. Feil, Lattice-Ordered Groups; An Introduction, D. Reidel, 1988.
\bibitem{BCGJT} 
P. Bahls, J. Cole, N. Galatos, P. Jipsen, C Tsinakis, Cancellative residuated lattices, Algebra Universalis,
(2003) Vol. 50, Issue 1, pp 83-106.
\bibitem{StIMTL} P. Baldi, A. Ciabattoni, F. Gulisano, Standard completeness for extensions of IMTL, 2017 IEEE International Conference on Fuzzy Systems (FUZZ-IEEE), 2017, pp. 1--6.
\bibitem{BlokRaftery} W. Blok and J. G. Raftery, Varieties of commutative residuated integral pomonoids and their residuation subreducts, Journal of Algebra 190 (1997), 280-328.
\bibitem{BonzioPlonka} S. Bonzio, Dualities for P\l{}onka sums, Logica Universalis, 12 (2018), 327--339.
\bibitem{BMB}
S. Bonzio, T. Moraschini, M. Pra Baldi, Logics of left variable inclusion and P\l{}onka sums of martices, Archive for Mathematical Logic, 60 (2021), 49--76
\bibitem{PriestleyIMTLis} L. M. Cabrer, S. A. Celani, Priestley dualities for some lattice-ordered algebraic structures, including MTL, IMTL and MV-algebras,  Central European Journal of Mathematics, 4 (4) (2006), 600--623.
\bibitem{CCkripke} L. M. Cabrer, S. A. Celani, Kripke Semantics for Monoidal T-norm based Logics MTL and IMTL, 
\bibitem{casari} E. Casari, Comparative logics and Abelian $\ell$-groups, in R. Ferro et al. (Eds.), Logic Colloquium 88, North Holland, Amsterdam, 1989, pp. 161- 190.
\bibitem{COM} R. L. O. Cignoli, I. M. L. D'Ottaviano,  D. Mundici, Algebraic Foundations of Many-Valued Reasoning, Trends in Logic, Vol. 7, Kluwer Academic Publishers, 1999.
\bibitem{IMTLCIG} R. Cignoli,  A.T. Torrens, Free Algebras in varieties of Glivenko MTL-algebras satisfying the equation $2(x^2) = (2x)^2$, Studia Logica 83 (2006), 157--181
\bibitem{FSystems} P. Cintula, P. H\'ajek, R. Hor\v{c}ik, Formal systems of fuzzy logic and their fragments, Annals of Pure and Applied Logic, 150(1-3) (2007), 40--65.
\bibitem{Dunn} J. M. Dunn. Algebraic completeness results for R mingle and its extensions. Journal of Symbolic  Logic, 35(1) (1970), 1--13. 
\bibitem{stIMTLeredeti} F. Esteva, J. Gispert, L. Godo, F. Montagna, On the standard and rational completeness of some axiomatic extensions of the Monoidal T-norm Logic, Studia Logica, 71(2) (2002), 199--226.
\bibitem{IMTLgamma} F. Esteva, L. Godo, Towards the generalization of Mundici’s $\Gamma$ functor to IMTL algebras: The linearly ordered case. In: Aguzzoli S., Ciabattoni A., Gerla B., Manara C., Marra V. (eds) Algebraic and Proof-theoretic Aspects of Non-classical Logics. Lecture Notes in Computer Science, vol 4460. (2007) Springer, Berlin, Heidelberg.
\bibitem{Galatos:2005kh} N. Galatos, Minimal varieties of residuated lattices. Algebra Universalis, 52(2-3), 215--239. (2005). 
\bibitem{gjko} N. Galatos, P. Jipsen, T. Kowalski, H. Ono, Residuated Lattices: An Algebraic Glimpse at Substructural Logics, 532 pp., Elsevier, Amsterdam, 2007.
\bibitem{GalRaf} N. Galatos, J.~G. Raftery, A category equivalence for odd Sugihara monoids and its applications, Journal of Pure and Applied Algebra, 216, 2177-2192 (2012)
\bibitem{Gal2015} N. Galatos, J.~G. Raftery, Idempotent residuated structures: Some category equivalences and their applications, Trans. Amer. Math. Soc. 367 (2015), 3189-3223 
\bibitem{GenMV} N. Galatos and C. Tsinakis, Generalized MV-algebras, Journal of Algebra 283 (2005),
254-291.
\bibitem{GJM} J. Gil-Férez, P. Jipsen, G. Metcalfe, Structure theorems for idempotent residuated lattices, Algebra Universalis, published online May 4, 2020, https://doi.org/10.1007/s00012-020-00659-5
\bibitem{IMT3} J. Gispert, A. Torrens, Axiomatic Extensions of IMT3 Logic, Studia Logica, 81(3) (2005), 311--324.
\bibitem{hajekBOOK} P. H\'ajek, Metamathematics of Fuzzy Logic, Trends in Logic, Vol. 4, Kluwer Academic Publishers, 1998.
\bibitem{HorcikAlgSem}
R. Hor\v cik. Algebraic semantics. In: Cintula, P., Hajek, P., Noguera, C. (eds.) Handbook of Mathematical Fuzzy Logic, pp. 283--353. College Publications (2011)
\bibitem{JS_amalg_dens}
S. Jenei, Amalgamation and densification in classes of involutive commutative residuated lattices, ArXiv:2012.14181 (2020)
\bibitem{JSIII} S. Jenei, Structure of left-continuous triangular norms with strong induced negations. (III) Construction and decomposition, Fuzzy Sets and Systems, 128:(2), (2002) 197-208. 
\bibitem{InvOrdSum} S. Jenei, Co-rotation, co-rotation-annihilation, and involutive ordinal sum constructions of residuated semigroups, Proceedings of the 19$\rm^{th}$ International Conference on Logic for Programming, Artificial Intelligence and Reasoning, 2013, Stellenbosch, South Africa, paper 73
\bibitem{Jenei_Hahn_err} S. Jenei, {Correction to \lq\lq The Hahn embedding theorem for a class of residuated semigroups\rq\rq} [Studia Logica 108 (2020), 1161--1206, Studia Logica (https://doi.org/10.1007/s11225-020-09933-y)
\bibitem{GroupReprArXiv} S. Jenei, Group representation for even and odd involutive commutative residuated chains,  arXiv:1910.01404 (2019)
\bibitem{JS_Hahn} S. Jenei, The Hahn embedding theorem for a class of residuated semigroups, Studia Logica, 108 (2020), 1161--1206
\bibitem{JF} P. Jipsen, F. Montagna, Embedding theorems for classes of GBL-algebras, Journal of Pure and Applied Algebra, 214, 1559-1575 (2010)
\bibitem{JKpsBCK} Jan K\"uhr, Representable pseudo-BCK-algebras and integral residuated lattices, Journal of Algebra, 317 (1), 2007, 354-364.
\bibitem{InverseSemigroups} M.V. Lawson, Inverse Semigroups -- The Theory of Partial Symmetries, 
World Scientific,1998, 428 pp., ISBN: 978-981-281-668-9
\bibitem{lawson} J. Lawson, Fifty Years of Topological Algebra, Seminar Sophus Lie XXXIV, Workshop in honor of Karl Hofmann's 75th Birthday, Darmstadt, October 5-6, 2007
\bibitem{IMTLstates}
L. Liu, X. Zhang, States on finite linearly ordered IMTL-algebras, Soft Computing 15 (2011), 2021--2028  
\bibitem{GeorgePhD} G. Metcalfe, Proof Theory for Propositional Fuzzy Logics, PhD thesis, King’s College London (2003)
\bibitem{MetMontSubstr} G. Metcalfe, F. Montagna, Substructural fuzzy logics,  Journal of Symbolic  Logic, 72 (3) (2007), 834--864.
\bibitem{MeyerAbelian} R.~K. Meyer, J.~K. Slaney, Abelian logic (from A to Z), in R. Routley et al. (Eds.), Paraconsistent Logic: Essays on the Inconsistent, Philosophia, Munich, 1989, pp. 245-288.
\bibitem{MoTs}  F. Montagna, C. Tsinakis, Ordered groups with a conucleus, Journal of Pure and Applied Algebra, 21481) (2010), 71.-88
\bibitem{Mos57} P.~S. Mostert, A.~L. Shields, On the structure of semigroups on a compact manifold with boundary, Ann.\ Math., {65}, 117-143 (1957)
\bibitem{Gamma} D. Mundici, Interpretation of AF C$^*$-algebras in \L ukasiewicz sentential calculus, 
J. Funct. Anal., 65 (1) (1986), pp. 15--63.
\bibitem{PerfectAgain} C. Noguera, F. Esteva, J. Gispert, On Some Varieties of MTL-algebras,
Logic Journal of the IGPL 13(4), (2005),443--466.
\bibitem{Perfect}
C. Noguera, F. Esteva, J. Gispert, 
Perfect and bipartite IMTL-algebras and disconnected rotations of prelinear semihoops,
Archive for Mathematical Logic, 44 (2005), 869--886.
\bibitem{Olson} J. S. Olson, Free representable idempotent commutative residuated lattices,
International Journal of Algebra and Computation, 18(8) (2008), 1365--1394.
\bibitem{AbelianNow} F. Paoli, M. Spinks, R. Verodd, Abelian Logic and the Logics of Pointed Lattice-Ordered Varieties, Logica Universalis, 2(2), pp. 209-233 (2008)
\bibitem{plonka} J. P\l{}onka, On a method of construction of algebras, Fundam. Math.  61(2) 183--189, (1967)
\bibitem{PlonkaProgram} H. Puhlmann, The snack powerdomain for database semantics. In: A. M. Borzyszkowski, S. Sokolowski, (eds.) Mathematical Foundations of Computer Science, pp. 650–659. Springer, Berlin (1993)
\bibitem{IdemP} J.~G. Raftery, Representable idempotent commutative residuated lattices, Transactions of the American Mathematical Society 359: 4405-4427, 2007
\bibitem{Saito} T. Sait$\hat{\rm o}$, Ordered inverse semigroups, Transactions of the American Mathematical Society, 153 (1971), 99--138.
\bibitem{IUL?} SanMin Wang, A Proof of the Standard Completeness for the Involutive Uninorm Logic, Symmetry, 11 (2019), 445; doi:10.3390/sym11040445 
\end{thebibliography}
\end{document}